\documentclass[12pt]{article}
\usepackage[margin=1in]{geometry}  
\usepackage{graphicx}              
\usepackage{amsmath}               
\usepackage{amsthm}                
\usepackage[ansinew]{inputenc}
\usepackage{array,color}
\usepackage{amsxtra}
\usepackage{amstext}
\usepackage{latexsym}
\usepackage{dsfont}
\usepackage[all]{xy}
\usepackage{amsmath,amsfonts,amssymb}
\usepackage{soul}

\usepackage{marginnote}

\usepackage{enumitem}  
\usepackage{calc}
\setlist{labelindent=1pt,itemsep=.5em}
\setlist[itemize]{leftmargin=1.2cm}
\setlist[enumerate]{itemindent=0em,leftmargin=1.2cm}
\setlist[enumerate,1]{label={\upshape(\roman*)}}

\usepackage{cite}

\usepackage{hyperref}
\hypersetup{
colorlinks   = true,
citecolor    = black,
linkcolor =black,
urlcolor=black,
}

\hypersetup{hidelinks}

\usepackage{authblk}

\newcommand{\email}[1]{%
    \normalsize\href{mailto:#1}{\color{black}{#1} }}

\allowdisplaybreaks

\makeatletter
\newcommand{\subjclass}[2][2020]{%
  \let\@oldtitle\@title%
  \gdef\@title{\@oldtitle\footnotetext{#1 \emph{Mathematics subject classification}: #2}}%
}
\newcommand{\keywords}[1]{%
  \let\@@oldtitle\@title%
  \gdef\@title{\@@oldtitle\footnotetext{\emph{Keywords}: #1}}%
}
\makeatother

\newtheorem{thm}{Theorem}[section]
\newtheorem{cor}[thm]{Corollary}
\newtheorem{lem}[thm]{Lemma}
\newtheorem{prop}[thm]{Proposition}
\theoremstyle{definition}
\newtheorem{defn}[thm]{Definition}
\theoremstyle{remark}
\newtheorem{rmk}[thm]{Remark}
\theoremstyle{remark}

\newtheorem{ex}[thm]{Example}

\numberwithin{equation}{section}

\title{Transposed Hom-Poisson and Hom-pre-Lie Poisson
algebras and bialgebras}

\author[1,2]{Ismail Laraiedh}
\author[3]{Sergei Silvestrov}
\affil[1]{\Affilfont Departement of Mathematics, Faculty of Sciences,
\authorcr \Affilfont Sfax University, Box 1171, 3000 Sfax, Tunisia
\authorcr \Affilfont
\email{ismail.laraiedh@gmail.com}}
\affil[2]{Departement of Mathematics,
\authorcr \Affilfont College of Sciences and Humanities Al Quwaiiyah,
\authorcr \Affilfont Shaqra University, Kingdom of Saudi Arabia
\authorcr \Affilfont
\email{ismail.laraiedh@su.edu.sa}}
\affil[3]{\Affilfont Division of Mathematics and Physics,
\authorcr \Affilfont School of Education, Culture and Communication,
\authorcr \Affilfont M\"{a}lardalen University, Box 883, 72123 V{\"a}ster{\aa}s, Sweden
\authorcr \Affilfont
\email{sergei.silvestrov@mdh.se}}


\date{}

\begin{document}
\maketitle

\abstract{
The notions of transposed Hom-Poisson and Hom-pre-Lie Poisson algebras are introduced. Their bimodules and matched pairs are defined and the relevant properties and theorems are given. The notion of Manin triple of transposed Hom-Poisson algebras is introduced, and its equivalence to the transposed Hom-Poisson bialgebras is investigated. The notion of $\mathcal{O}$-operator is exploited to illustrate the relations existing between transposed Hom-Poisson and Hom-pre-Lie Poisson algebras.
}

\section{Introduction}
A Poisson algebra $(A, \mu, \{\cdot,\cdot\})$ consists of a commutative associative algebra $(A, \mu)$ together with a Lie algebra structure $\{\cdot,\cdot\}$, that satisfies, for all $x,y,z\in A$, the Leibniz identity
\begin{equation*}
    \{\mu(x,y),z\}=\mu(\{x,z\},y)+\mu(x,\{y,z\}).
\end{equation*}
Poisson algebras appear in many areas of mathematics and physics being important for integrable systems, Poisson geometry, Poisson brackets and Poisson manifolds, geometry of smooth loops, vertex algebras, quantization theory, deformation quantization, classical mechanics, Hamiltonian mechanics, quantum mechanics, quantum field theory, Lie theory and operads \cite{AgoreMilitaru2015:JacobiPoissonalgs,Arn78,BLLM2017,BhaskaraViswanath88,ChariPressley94:GuideQGr,
Dirac64,Drinfeld87,Dzhumadildaev2002:identderivJacalg,
Frenkel,Fresse,Gerstenhaber64,GinzburgKaledin04,GozeRemm2008:Poisalgnonassalg,GrabowskiMarmo99:NambuPoissNambuJacbrackets,GurseyTze96,Huebschmann90,
KarasevMaslov:NonlinPoissonbrquantiz,Kontsevich03,Li99,Lichnerowicz77,MarklRemm2006,MishchenkoPetrogrRegev2007:PoisPIalg,OdA,Pol97,Xu:noncomPoissonalg94,Vaisman1,Weinstein77}.
Poisson algebras led to development of other fundamental structures such as noncommutative Poisson algebras \cite{Xu:noncomPoissonalg94}, Jacobi algebras (generalized Poisson algebras) \cite{AgoreMilitaru2015:JacobiPoissonalgs,CantariniKac2007,Dzhumadildaev2002:identderivJacalg,Kirillov76,Lichnerowicz77}, Gerstenhaber algebras and
Lie-Rinehart algebras \cite{Gerstenhaber63,KosmannSchwarzbach96,LodayVallette2012,Rinehart63} and Novikov-Poisson algebras \cite{BokutChenZhang2016,Xu:NovikovPoissonalg97}
arising from Novikov algebra in connection with the Poisson
brackets of hydrodynamic type \cite{BalinskiiNovikov85} and Hamiltonian operators
in the formal variational calculus \cite{GelfandDorfman79}.

A transposed Poisson algebra \cite{BaiBaiGuoWu2020:transpPoisNovPois3Lie} is a triple $(A, \mu, \{\cdot,\cdot\})$ consists of a commutative associative algebra $(A, \mu)$ and
a Lie algebra $(A,\{\cdot,\cdot\})$, satisfying, for all $x,y,z\in A$, the transposed Leibniz identity
\begin{equation}\label{trans Poisson}
    2\mu(z,\{x,y\})=\{\mu(z, x),y\}+\{x,\mu(z\cdot y)\}.
\end{equation}
While both the Poisson
algebra and transposed Poisson algebra are built from a commutative associative operation
and a Lie operation, the intersection of the compatibility conditions of
the Poisson algebra and of the transposed Poisson algebra is trivial, meaning that the structures
of transposed Poisson algebras and Poisson algebras are incompatible in this sense   \cite{BaiBaiGuoWu2020:transpPoisNovPois3Lie}.

The theory of Hom-algebras has been initiated in \cite{HartwigLarSil:defLiesigmaderiv, LarssonSilvJA2005:QuasiHomLieCentExt2cocyid,LarssonSilv:quasiLiealg} motivated by quasi-deformations of Lie algebras of vector fields, in particular q-deformations of Witt and Virasoro algebras. Hom-Lie algebras and more general quasi-Hom-Lie algebras were introduced first by Hartwig, Larsson and Silvestrov in  \cite{HartwigLarSil:defLiesigmaderiv} where a general approach to discretization of Lie algebras of vector fields using general twisted derivations ($\sigma$-deriva\-tions) and a general method for construction of deformations of Witt and Virasoro type algebras based on twisted derivations have been developed. The general quasi-Lie algebras, containing the quasi-Hom-Lie algebras and Hom-Lie algebras as subclasses, as well their graded color generalization, the color quasi-Lie algebras including color quasi-hom-Lie algebras, color hom-Lie algebras and their special subclasses the quasi-Hom-Lie superalgebras and hom-Lie superalgebras, have been first introduced in \cite{HartwigLarSil:defLiesigmaderiv,LarssonSilvJA2005:QuasiHomLieCentExt2cocyid,LarssonSilv:quasiLiealg,LSGradedquasiLiealg,LarssonSilv:quasidefsl2,SigSilv:CzechJP2006:GradedquasiLiealgWitt}.
Subsequently, various classes of Hom-Lie admissible algebras have been considered in \cite{ms:homstructure}. In particular, in \cite{ms:homstructure}, the Hom-associative algebras have been introduced and shown to be Hom-Lie admissible, that is leading to Hom-Lie algebras using commutator map as new product, and in this sense constituting a natural generalization of associative algebras as Lie admissible algebras leading to Lie algebras using commutator map. Furthermore, in \cite{ms:homstructure}, more general $G$-Hom-associative algebras including Hom-associative algebras, Hom-Vinberg algebras (Hom-left symmetric algebras), Hom-pre-Lie algebras (Hom-right symmetric algebras), and some other Hom-algebra structures, generalizing $G$-associative algebras, Vinberg and pre-Lie algebras respectively, have been introduced and shown to be Hom-Lie admissible, meaning that for these classes of Hom-algebras, the operation of taking commutator leads to Hom-Lie algebras as well. Also, flexible Hom-algebras have been introduced, connections to Hom-algebra generalizations of derivations and of adjoint maps have been noticed, and some low-dimensional Hom-Lie algebras have been described.
In Hom-algebra structures, defining algebra identities are twisted by linear maps.
Since the pioneering works \cite{HartwigLarSil:defLiesigmaderiv,LarssonSilvJA2005:QuasiHomLieCentExt2cocyid,
LarssonSilv:quasiLiealg,LSGradedquasiLiealg,LarssonSilv:quasidefsl2,ms:homstructure}, Hom-algebra structures have developed in a popular broad area with increasing number of publications in various directions.
Hom-algebra structures include their classical counterparts and open new broad possibilities for deformations, extensions to Hom-algebra structures of representations, homology, cohomology and formal deformations, Hom-modules and hom-bimodules, Hom-Lie admissible Hom-coalgebras, Hom-coalgebras, Hom-bialgebras, Hom-Hopf algebras, $L$-modules, $L$-comodules and Hom-Lie quasi-bialgebras, $n$-ary generalizations of BiHom-Lie algebras and BiHom-associative algebras and generalized derivations, Rota-Baxter operators, Hom-dendriform color algebras, Rota-Baxter bisystems and covariant bialgebras, Rota-Baxter cosystems, coquasitriangular mixed bialgebras, coassociative Yang-Baxter pairs, coassociative Yang-Baxter equation and generalizations of Rota-Baxter systems and algebras, curved $\mathcal{O}$-operator systems and their connections with tridendriform systems and pre-Lie algebras, BiHom-algebras, BiHom-Frobenius algebras and double constructions, infinitesimal BiHom-bialgebras and Hom-dendriform $D$-bialgebras, Hom-algebras have been considered
\cite{AbdaouiMabroukMakhlouf,
AmmarEjbehiMakhlouf:homdeformation,
AttanLaraiedh:2020ConstrBihomalternBihomJordan,
Bakayoko:LaplacehomLiequasibialg,
Bakayoko:LmodcomodhomLiequasibialg,
BakBan:bimodrotbaxt,
BakyokoSilvestrov:HomleftsymHomdendicolorYauTwi,
BakyokoSilvestrov:MultiplicnHomLiecoloralg,
BenHassineChtiouiMabroukNcib19:CohomLiedeformBiHomleftsym,
BenMakh:Hombiliform,
BenAbdeljElhamdKaygorMakhl201920GenDernBiHomLiealg,
CaenGoyv:MonHomHopf,
ChtiouiMabroukMakhlouf1,
ChtiouiMabroukMakhlouf2,
DassoundoSilvestrov2021:NearlyHomass,
EbrahimiFardGuo08,
GrMakMenPan:Bihom1,
HassanzadehShapiroSutlu:CyclichomolHomasal,
HounkonnouDassoundo:centersymalgbialg,
HounkonnouHoundedjiSilvestrov:DoubleconstrbiHomFrobalg,
HounkonnouDassoundo:homcensymalgbialg,
kms:narygenBiHomLieBiHomassalgebras2020,
Laraiedh1:2021:BimodmtchdprsBihomprepois,
LarssonSigSilvJGLTA2008:QuasiLiedefFttN,
LarssonSilvJA2005:QuasiHomLieCentExt2cocyid,
LarssonSilv:quasidefsl2,
LarssonSilvestrovGLTMPBSpr2009:GenNComplTwistDer,
LiuMakhMenPan:RotaBaxteropsBiHomassalg,
MaMakhSil:CurvedOoperatorSyst,
MaMakhSil:RotaBaxbisyscovbialg,
MaMakhSil:RotaBaxCosyCoquasitriMixBial,
MaZheng:RotaBaxtMonoidalHomAlg,
MabroukNcibSilvestrov2020:GenDerRotaBaxterOpsnaryHomNambuSuperalgs,
Makhl:HomaltHomJord,
Makhlouf2010:ParadigmnonassHomalgHomsuper,
MakhloufHomdemdoformRotaBaxterHomalg2011,
MakhSil:HomHopf,
MakhSilv:HomDeform,
MakhSilv:HomAlgHomCoalg,
MakYau:RotaBaxterHomLieadmis,
RichardSilvestrovJA2008,
RichardSilvestrovGLTbnd2009,
SaadaouSilvestrov:lmgderivationsBiHomLiealgebras,
ShengBai:homLiebialg,
Sheng:homrep,
SigSilv:GLTbdSpringer2009,
SilvestrovParadigmQLieQhomLie2007,
SilvestrovZardeh2021:HNNextinvolmultHomLiealg,
QSunHomPrealtBialg,
Yau:ModuleHomalg,
Yau:HomEnv,
Yau:HomHom,
Yau:HombialgcomoduleHomalg,
Yau:HomYangBaHomLiequasitribial,
YauHomMalcevHomalternHomJord}.

A Hom-type generalization of Poisson algebras, called Hom-Poisson algebras, was introduced in \cite{MakhSilv:HomDeform}. We aim in this paper to introduce the notions of transposed Hom-Poisson and Hom-pre-Lie Poisson algebras. The concept of bimodules and matched pairs are constructed, and we give some related results. Next, we introduce notion of Manin triple of transposed Hom-Poisson algebras and transposed Hom-Poisson bialgebras and then give the equivalence between them. Also,
we introduce and study the notion of an $\mathcal{O}$-operator of transposed Hom-Poisson and Hom-pre-Lie Poisson algebras.

The paper is organized as follows. In Section \ref{sec2}, we introduce the notion of transposed Hom-Poisson algebra, and explain connections between transposed Hom-Poisson algebras and Hom-
Poisson algebras. Moreover, we define the notion of a bimodule of a transposed Hom-Poisson algebra and
provide some properties. In Section \ref{sec3},
we introduce the notion of Manin triples for  transposed Hom-Poisson algebras and transposed Hom-Poisson bialgebras. Also, the equivalences between their Manin triples are established. In Section \ref{sec4}, we establish definition of Hom-pre-Lie Poisson algebra and we develop
some construction theorems. Their bimodule and matched pair are defined and their related
relevant properties are also given. In Section \ref{sec5}, we introduce and study the notion of an $\mathcal{O}$-operator of transposed Hom-Poisson algebras generalizing the notion of Rota-Baxter operators. Their relationship with Hom-pre-Lie Poisson algebras is also described.

\section{Transposed Hom-Poisson algebras}\label{sec:HomtransposedPoissonalgebras}\label{sec2}
Throughout this paper $\mathbb{K}$ is an algebraically closed field
of characteristic $0$ and $A$ is a $\mathbb{K}$-vector space.
We refer to the standard one-to-one correspondence between linear maps
$\mu : V_1\otimes\dots\otimes V_n \rightarrow W$ and multilinear maps
$\mu : V_1\times\dots\times V_n \rightarrow W$ given by
$\mu(v_1,\dots,v_n)=\mu(v_1\otimes\dots\otimes v_n)$, whenever the same notation is used for these maps.

In this section, we introduce the notion of transposed Hom-Poisson algebra with a connection to commutative Hom-associative algebra and Hom-Lie algebra. Next, we study the notion of bimodule and  matched pair of transposed Hom-Poisson algebras.

\begin{defn}
A Hom-module is a pair $(M,\alpha_M)$ consisting of a
$\mathbb{K}$-module $M$ and a linear map $\alpha_M:
M\longrightarrow M$. A morphism $f: (M,\alpha_M)\rightarrow
(N,\alpha_N)$ of Hom-modules is a linear map
 $f: M\longrightarrow N$ such that $f\alpha_M=\alpha_N f$.
\end{defn}

\begin{defn}
A Hom-algebra is a triple $(A,\mu,\alpha)$ where $(A,\alpha)$ is
a Hom-module and $\mu : A^{\otimes 2} \rightarrow A$ is a linear map.
\end{defn}

\begin{defn}
A Hom-algebra $(A, \mu,\alpha)$ is called multiplicative, if
the linear map $\alpha:A\to A$ is moreover an algebra homomorphism of the algebra $(A,\mu)$, that is,  $$\alpha\circ\mu=\mu\circ(\alpha\otimes\alpha),$$
or equivalently, in multilinear maps notation,
$\alpha\circ\mu (x,y)=\mu(\alpha(x),\alpha(y)),$
for $x,y\in A.$
\end{defn}

In any Hom-algebra $(A,\mu,\alpha)$, Hom-associator ($\alpha$-associator) is defined by
$$ as_{A}(x,y,z)=\mu(\mu(x,y),\alpha(z))-\mu(\alpha(x),\mu(y,z))$$
for all $x,y,z \in A$. Hom-associators are generalizations to Hom-algebras of the associators in algebras (when $\alpha=id$).
\begin{defn}
A Hom-associative algebra is a Hom-algebra
$(A,\mu,\alpha)$ satisfying
\begin{align}\label{Homass:homassociator}
    as_{A}(x,y,z)=0, &&  \mbox{(Hom-associativity)}
\end{align}
or equivalently,
$\mu(\alpha(x),\mu(y,z))=\mu(\mu(x,y),\alpha(z)),$
for $x,y,z\in A$.
\end{defn}

\begin{defn}
A commutative Hom-associative algebra is a Hom-algebra
$(A,\mu,\alpha)$ satisfying, for all $x,y,z\in A$,
\begin{align}\label{Homalg:Commutativityidentity}
&\mu(x,y)=\mu(y,x), & \mbox{(Commutativity)}
    \\
    &as_{A}(x,y,z)=0. & \mbox{(Hom-associativity)}
    \end{align}
\end{defn}
\begin{ex}
	Consider a two-dimensional $\mathbb{K}$-linear space $ A$ with basis $\{ e_1, e_2\}$.
	\begin{enumerate}[label=\arabic*)]
		\item
		The triple $(A, \mu,\alpha)$ with the bilinear product and linear map defined on $A$ by
\begin{align*}
\mu(e_1, e_1)= -e_1,\quad & \mu(e_1, e_2)= \mu(e_2,e_1)=e_2,\quad \mu(e_2, e_2)=e_1, \\
& \alpha(e_1)=e_1,\quad \alpha(e_2)=-e_2,
\end{align*}
is a commutative Hom-associative  algebra.
		\item The triple $(A, \mu,\alpha)$ with the bilinear product and linear map defined on $A$ by
\begin{align*}
\mu(e_1, e_1)= e_1, &\quad \mu(e_1, e_2)= \mu(e_2,e_1)=0, \quad \mu(e_2, e_2)=e_2, \\
& \quad \alpha(e_1)=e_1, \quad \alpha(e_2)=0.
\end{align*}
is a commutative Hom-associative  algebra.
\end{enumerate}
\end{ex}
\begin{ex}
	Consider a three-dimensional $\mathbb{K}$-linear space $ A$ with basis $\{ e_1, e_2, e_3\}$.
	\begin{enumerate}[label=\arabic*)]
		\item
		The triple $(A, \mu,\alpha)$ with the bilinear product and linear map defined on $A$ by
		\begin{align*}
		e_1\cdot e_1&=e_1, \quad e_2\cdot e_2= e_2+e_3, \quad e_3\cdot e_3=e_2+e_3\\
		e_2\cdot e_3&=e_2+e_3, \quad e_3\cdot e_2=e_2+e_3, \\
\alpha(e_1)&=e_1,
		\end{align*}
		where the non written products
and images of $\alpha$ are  zero,
is a commutative Hom-associative  algebra.
		\item
		The triple $(A, \mu,\alpha)$ with the bilinear product and linear map defined on $A$ by
		\begin{align*}
        e_1\cdot e_1 &=p_1e_1, \quad e_2\cdot e_2= p_2e_2 , \quad e_3\cdot e_3=p_3 e_3 ,\\
		\alpha(e_1)&=e_1, \quad \alpha(e_2)=e_2,
\end{align*}
		where the non written products and images of $\alpha$ are zero,
is a commutative Hom-associative  algebra.
	\end{enumerate}	
\end{ex}
\begin{defn}
A derivation of a commutative Hom-associative algebra $(A, \mu,\alpha)$ is a linear map
$D : A \rightarrow A$ satisfying for all $x, y \in A$,
\begin{align}
    \alpha\circ D&=D\circ \alpha, \label{derivation1}\\
    D(x\cdot y)&=D(x)\cdot y+x\cdot D(y).
\end{align}
\end{defn}
We will use in this article a definition of bimodule of a commutative Hom-associative algebras
including Hom-module map condition \eqref{Cond2}, while we note that there are also other definitions of Hom-modules and Hom-bimodules of Hom-associative algebras, for example the more general notions requiring only \eqref{Cond1},
\cite{Bakayoko:LmodcomodhomLiequasibialg,
BakBan:bimodrotbaxt,MakhSil:HomHopf,
MakhSilv:HomDeform,
HassanzadehShapiroSutlu:CyclichomolHomasal,
Yau:ModuleHomalg,
Yau:HombialgcomoduleHomalg}.
\begin{defn}
Let $(A, \cdot, \alpha)$ be a commutative Hom-associative algebra and let $(V, \beta)$ be a Hom-module. Let $ s: A \rightarrow gl(V) $ be a linear map. The triple $(s, \beta, V)$ is called a bimodule of $(A, \cdot, \alpha)$ if for all $ x, y \in  A, v \in V $,
\begin{align}
\label{Cond1}
s(x\cdot y)\beta(v)=s(\alpha(x))s(y)v, \\
\label{Cond2}
\beta(s(x)v) = s(\alpha(x))\beta(v).
\end{align}
\end{defn}
\begin{prop}\label{ass1}
Let $(s, \beta, V)$ be a bimodule of a commutative Hom-associative algebra $(A, \cdot, \alpha)$.
Then, the direct sum $A \oplus V$ of vector spaces is a commutative Hom-associative algebra with multiplication defined for all $ x_{1}, x_{2} \in  A, v_{1}, v_{2} \in V$ by
\begin{align*}
(x_{1} + v_{1}) \cdot' (x_{2} + v_{2}) &=x_{1} \cdot x_{2} + (s(x_{1})v_{2} + s(x_{2})v_{1}),\cr
(\alpha\oplus\beta)(x_{1} + v_{1}) &=\alpha(x_{1}) + \beta(v_{1}).
\end{align*}
\end{prop}
\begin{proof}
For all $x_1,x_2\in A,~v_1,v_2\in V$,
\begin{align*}
&(x_1+v_1)\cdot'(x_2+v_2)=x_1\cdot x_2+s(x_1)v_2+s(x_2)v_1\\
    &\quad =x_2\cdot x_1+s(x_2)v_1+s(x_1)v_2 = (x_2+v_2)\cdot'(x_1+v_1), \\
&\Big[(x_1+v_1)\cdot'(x_2+v_2)\Big]\cdot'(\alpha+\beta)(x_3+v_3) \\
    &\quad =\Big[x_1\cdot x_2+s(x_1)v_2+s(x_2)v_1\Big]\cdot'(\alpha(x_3)+\beta(v_3))\\
    &\quad =
    (x_1\cdot x_2)\cdot \alpha(x_3)+ s(x_1\cdot x_2)\beta(v_3)+s(\alpha(x_{3}))(s(x_1)v_2+s(x_2)v_1)\\
    &\quad = \alpha(x_1)\cdot(x_2\cdot x_3)+s(\alpha(x_1))s(x_2)\beta(v_3)+s(x_3\cdot x_1)v_2+s(x_3\cdot x_2)v_1~~(by \eqref{Cond1})\\
    &\quad =(\alpha+\beta)(x_1+v_1)\cdot'\Big[(x_2+v_2)\cdot '(x_3+v_3)\Big].
\end{align*}
Then, $(A \oplus V, \cdot', \alpha + \beta)$ is a commutative Hom-associative algebra.
\end{proof}
The commutative Hom-associative algebra constructed in Proposition \ref{ass1} is denoted
by $(A \oplus V, \cdot', \alpha + \beta)$ or $A \times_{s, \alpha, \beta} V.$
\begin{thm}\label{matched ass}
Let $\mathcal{A}=(A,\cdot_A,\alpha)$ and $\mathcal{B}=(B,\cdot_B,\beta)$ be two commutative Hom-associative algebras. Suppose that there are linear maps $s_A:A\rightarrow gl(B)$
and $s_B:B\rightarrow gl(A)$ such that $(s_A,\beta,B)$ is a bimodule of $\mathcal{A}$, and $(s_B,\alpha,A)$ is a bimodule of $\mathcal{B}$ satisfying, for any $x,y\in A,~a,b\in B$,
\begin{align}
\label{comm1}
    s_A(\alpha(x))(a\cdot_B b)&=\beta(a)\cdot_B(s_A(x)b)+s_A(s_B(b)x)\beta(x), \\
\label{comm2}
   (s_A(x)a)\cdot_B\beta(b)+s_A(s_B(a)x)\beta(b)&=\beta(a)\cdot_B(s_A(x)b)+s_A(s_B(b)x)\beta(a), \\
\label{comm3}
    s_B(\beta(a))(x\cdot_A y)&=\alpha(x)\cdot_A(s_B(a)y)+s_B(s_A(y)a)\alpha(a), \\
\label{comm4}
   (s_B(a)x)\cdot_A\alpha(y)+s_B(s_A(x)a)\alpha(y)&=\alpha(x)\cdot_A(s_B(a)y)+s_B(s_A(y)a)\alpha(x).
\end{align}
Then, $(A,B,s_A,\beta,s_B,\alpha)$ is called a matched pair of
commutative Hom-associative algebras. In this case, there is a commutative Hom-associative algebra structure on the direct sum
$A\oplus B$ of the underlying vector spaces of $A$ and $B$ given by
$$\begin{array}{llllll}
(x + a) \cdot (y + b)&=&x \cdot_A y + (s_A(x)b + s_A(y)a)+a \cdot_B b + (s_B(a)y + s_B(b)x),\cr
(\alpha\oplus\beta)(x + a)&=&\alpha(x) + \beta(a).
\end{array}$$
\end{thm}
\begin{proof}
For any $x,y,z\in A$ and $a,b,c\in B$,
\begin{align*}
&(\alpha+\beta)(x+a)\cdot((y+b)\cdot(z+c))\\
&\quad =(\alpha(x)+\beta(a))(y\cdot_A z+s_B(b)z+s_B(c)y+b\cdot c+s_A(y)c+s_A(z)b)\\
&\quad =\alpha(x)\cdot_A(y\cdot_A z)+\alpha(x)\cdot_A s_B(b)z+\alpha(x)\cdot_A s_B(c)y+s_B(\beta(a))(y\cdot_A z)\\
&\quad \quad + s_B(\beta(a))s_B(b)z+s_B(\beta(a))s_B(c)y+s_B(b\cdot_B c)\alpha(x)+s_B(s_A(y)c)\alpha(x)\\
&\quad \quad +s_B(s_A(z)b)\alpha(x)+\beta(a)\cdot_B(b\cdot_B c)+\beta(a)\cdot_B s_A(y)c+\beta_A(\alpha(x))s_A(y)c\\
&\quad \quad +s_A(\alpha(x))s_A(z)b+s_A(y\cdot_A z)\beta(a)+s_A(s_A(b)z)\beta(a)+s_A(s_B(c)y)\beta(a), \\
&((x+a)\cdot(y+b))\cdot(\alpha+\beta)(z+c)\\
&\quad =(x\cdot_A y+s_B(a)y+s_B(b)x+a\cdot_B b+s_A(x)b+s_A(y)a)\cdot(\alpha(z)+\beta(c))\\
&\quad =(x\cdot_A y)\cdot_A\alpha(z)+s_B(a)y\cdot_A\alpha(z)+s_B(b)x\cdot_A\alpha(z)+s_B(a\cdot_B b)\alpha(z)\\
&\quad \quad +s_B(s_A(x)b)\alpha(z)+s_B(s_A(y)a)\alpha(z)+s_B(\beta(c))(x\cdot_A y)+s_A(\beta(c))s_B(a)y\\
&\quad \quad +s_B(\beta(c))s_B(b)x+(a\cdot_B b)\cdot_B\beta(c)+(s_A(x)b)\cdot_B\beta(c)+(s_A(y)a)\cdot_B\beta(c)\\
&\quad \quad +r_A(\alpha(z))(a\cdot_B b)+s_A(\alpha(z))(s_A(x)b)+s_A(\alpha(z))(s_A(y)a)+s_A(x\cdot_A y)\beta(c)\\
&\quad \quad +s_A(s_B(a)y)\beta(c)+(s_B(b)x)\beta(c).
\end{align*}
Then, by \eqref{Homalg:Commutativityidentity}, \eqref{Homass:homassociator} and \eqref{comm1}-\eqref{comm4}, we deduce that $(\alpha+\beta)(x+a)\cdot((y+b)\cdot(z+c))=((x+a)\cdot(y+b))\cdot(\alpha+\beta)(z+c)$.
\end{proof}
\begin{defn}[\cite{HartwigLarSil:defLiesigmaderiv}]
A Hom-Lie algebra is a triple $(A,[\cdot,\cdot],\alpha)$ consisting
of a linear space $A$, bilinear map (bilinear product)
$[\cdot,\cdot]:A\otimes A\rightarrow A$, and a linear map
$\alpha:A\rightarrow A$ satisfying, for all $x, y, z\in  A$,
\begin{align} \label{eq_hom_skewsym_identity}
[x,y]=-[y,x],  & \quad \quad \mbox{(Skew-symmetry)}\\
\label{eq_hom_jacobi_identity}
[\alpha(x),[y,z]]+[\alpha(y), [z,x]]+ [\alpha(z), [x,y]] = 0. & \quad \quad \mbox{\rm (Hom-Jacobi identity)}
\end{align}
\end{defn}
\begin{ex}
Let $(A,\cdot,\alpha)$ be a commutative Hom-associative algebra and $D_1,D_2$
be commuting derivations (that is, $D_1D_2=D_2D_1$). Then, there is a Hom-Lie
algebra $(A,[\cdot,\cdot],\alpha)$, where for all $x,y\in A$,
\begin{equation}\label{eq:2der}
[x,y]=D_1(x)\cdot D_2(y)-D_1(y)\cdot D_2(x).
\end{equation}
\end{ex}
\begin{defn}[\cite{AmmarEjbehiMakhlouf:homdeformation,Sheng:homrep}]
Let $(A,[\cdot,\cdot],\alpha)$ be a Hom-Lie algebra and $(V,\beta)$ be a Hom-module. Let $\rho:A\rightarrow gl(V)$ be a linear map. The triple $(\rho,\beta,V)$ is called a representation of $A$ if for all $x,y\in A,~v\in V$,
\begin{eqnarray}\rho([x,y])\beta(v)&=&\rho(\alpha(x))\circ\rho(y)v-\rho(\alpha(y))\circ\rho(x)v,\label{repLie1}\\
\beta(\rho(x)v)&=&\rho(\alpha(x))\beta(v)\label{repLie3}.
\end{eqnarray}
\end{defn}

\begin{prop}\label{pro11}
Let $(\rho, \beta, V)$ be a representation of a Hom-Lie algebra
$(A, \{\cdot,\cdot\}, \alpha)$. Then, the direct sum $A \oplus V$ of vector
spaces is turned into a Hom-Lie algebra by defining the
multiplication in $A\oplus V $ by
\begin{eqnarray*}
[x_1+v_1,x_2+v_2] &=&\{x_1,x_2\}+\rho(x_1)v_2-\rho(x_2)v_1,\\
(\alpha\oplus\beta)(x_1+v_1) &=&\alpha(x_1)+\beta(v_1).
\end{eqnarray*}
\end{prop}
We denote such a Hom-Lie algebra by $(A \oplus V, [\cdot,\cdot], \alpha+ \beta),$
or $A \times_{\rho, \alpha, \beta} V.$

Now, we introduce the notion of matched pair of Hom-Lie algebra
\begin{thm}[\cite{ShengBai:homLiebialg}] \label{matched Lie}
Let $(A,\{\cdot,\cdot\}_A,\alpha)$ and $(B,\{\cdot,\cdot\}_{B},\beta)$ be two Hom-Lie algebras. Suppose that there are linear maps $\rho_A:A\rightarrow gl(B)$
and $\rho_B:B\rightarrow gl(A)$ such that $(\rho_A,\beta,B)$ is a representation of $A$ and $(\rho_B,\alpha,A)$ is a representation of $B$ satisfying for any $x,y\in A,~a,b\in B$,
\begin{eqnarray}
&\begin{array}{lll}
\rho_B(\beta(a))\{x,y\}_A&=&\{\rho_B(a)x,\alpha(y)\}_A+\{\alpha(x),\rho_B(a)(y)\}_A \\
&+&\rho_B(\rho_A(y)(a))\alpha(x)-\rho_B(\rho_A(x)(a))\alpha(y),
\end{array}
\label{Lie1}
\\
&\begin{array}{lll}
\rho_A(\alpha(x))\{a,b\}_B&=&\{\rho_A(x)a,\beta(b)\}_B+\{\beta(a),\rho_A(x)(b)\}_B \\
&+&\rho_A(\rho_B(b)(x))\beta(a)-\rho_A(\rho_B(a)(x))\beta(b).
\end{array}
\label{Lie2}
\end{eqnarray}
Then, $(A,B,\rho_A,\beta,\rho_B,\alpha)$ is called a matched pair of
Hom-Lie algebras. In this case, there is a Hom-Lie algebra structure on the vector
space $A\oplus B$ of the underlying vector spaces of $A$ and $B$ given by
\begin{align}
[x+a,y+b]&=\{x,y\}_A+\rho_A(x)b-\rho_A(y)a\\
&\quad \quad + \{a,b\}_B+\rho_B(a)y-\rho_B(b)x,\cr
(\alpha\oplus\beta)(x + a)&=\alpha(x) + \beta(a).
\end{align}
\end{thm}

Hom-Poisson algebras were introduced in \cite{MakhSilv:HomDeform}.
\begin{defn}[\cite{MakhSilv:HomDeform}]
A \emph{Hom-Poisson algebra} is a quadruple $(V,\mu, \{\cdot,\cdot\}, \alpha)$ consisting of
a vector space $V$, bilinear maps $\mu: V\times V \rightarrow V$ and
$\{\cdot, \cdot\}: V\times V \rightarrow V$, and a linear map $\alpha: V \rightarrow V$ satisfying,
 \begin{enumerate}
\item $(V,\mu, \alpha)$ is a commutative Hom-associative algebra,
\item $(V, \{\cdot,\cdot\}, \alpha)$ is a Hom-Lie algebra,
\item
for all $x, y, z$ in $V$,
\begin{equation}\label{CompatibiltyPoisson}
\{\alpha (x) , \mu (y,z)\}=\mu (\alpha (y),
\{x,z\})+ \mu (\alpha (z), \{x,y\}).
\end{equation}
\end{enumerate}
\end{defn}
Condition \eqref{CompatibiltyPoisson}, expressing the compatibility between the
multiplication and the Poisson bracket, can be reformulated equivalently as
\begin{equation}\label{CompatibiltyPoissonLeibform}
\{\mu(x,y),\alpha (z) \}=\mu (\{x,z\},\alpha
(y))+\mu (\alpha (x), \{y,z\}).
\end{equation}
\begin{ex}[\cite{MakhSilv:HomDeform}]
Let $\{e_1,e_2,e_3\}$ be a basis of a $3$-dimensional vector space $A$ over $\mathbb{K}$. The following multiplication $\cdot$, skew-symmetric bracket $\{\cdot,\cdot\}$ and linear map $\alpha$ on $V$ define a Hom-Poisson algebra:
\begin{eqnarray*}
& e_1\cdot e_1=e_1, \quad e_1\cdot e_2= e_3, \quad e_2\cdot e_1=e_3\\
& \{e_1,e_2\}=ae_2+be_3, \quad \{e_1,e_3\}=ce_2+de_3, \\
& \alpha(e_1)=\lambda_1 e_2+\lambda_2 e_3, \quad \alpha(e_2)=\lambda_3 e_2+\lambda_4 e_3, \quad	\alpha(e_3)=\lambda_5 e_2+\lambda_6 e_3,
\end{eqnarray*}
where $a,b,c,d,\lambda_1,\lambda_2,\lambda_3,\lambda_4,\lambda_5,\lambda_6$ are parameters in $\mathbb{K}.$
\end{ex}
\begin{prop}
Let $(A,\cdot,\alpha)$ be a commutative Hom-associative algebra, and let $D_1$ and $D_2$
be commuting derivations, $D_1D_2=D_2D_1$. For all $x,y\in A$, let
\begin{equation} 
\{x,y\}=D_1(x)\cdot D_2(y)-D_1(y)\cdot D_2(x).
\end{equation}
Then, $(A,\cdot,\{\cdot,\cdot\},\alpha)$ is a Hom-Poisson algebra.
\end{prop}
\begin{proof}
First, we prove that $(A,\{\cdot,\cdot\},\alpha)$ is a Hom-Lie algebra. For all $x,y,z\in A$,
\begin{align*}
&\{x,y\}=D_1(x)\cdot D_2(y)-D_1(y)\cdot D_2(x)=-(D_1(y)\cdot D_2(x)-D_1(x)\cdot D_2(y))\cr
&\quad \quad \quad   =-\{y,x\},\cr
  & \circlearrowleft_{x,y,z} \{\alpha(x),\{y,z\}\}=\circlearrowleft_{x,y,z}\Big( D_1(\alpha(x))\cdot D_2(\{y,z\})-D_1(\{y,z\})\cdot D_2(\alpha(x))\Big)\cr
  &\quad = \circlearrowleft_{x,y,z}\Big( D_1(\alpha(x))\cdot D_2(D_1(y)\cdot D_2(z)-D_1(z)\cdot D_2(y))\cr
  &\quad \quad \quad -D_1(D_1(y)\cdot D_2(z)-D_1(z)\cdot D_2(y))\cdot D_2(\alpha(x))\Big)\cr
  &\quad= \circlearrowleft_{x,y,z}\Big( \underbrace{D_1(\alpha(x))\cdot(D_2 D_1(y)\cdot D_{2}(z))}_{a_1(x,y,z)}+ \underbrace{D_1(\alpha(x))\cdot(D_1(y)\cdot D_{2}^{2}(z))}_{b_1(x,y,z)}\cr
  &\quad \quad \quad - \underbrace{D_1(\alpha(x))\cdot(D_2 D_1(z)\cdot D_2(y))}_{c_1(x,y,z)}- \underbrace{D_1(\alpha(x))\cdot(D_1(z)\cdot D_{2}^{2}(y))}_{d_1(x,y,z)}\cr
  &\quad \quad \quad -\underbrace{(D_{1}^{2}(y)\cdot D_2(z))\cdot D_2(\alpha(x))}_{e_1(x,y,z)}-\underbrace{(D_1(y)\cdot D_{1}D_2(z))\cdot D_{2}(\alpha(x))}_{f_1(x,y,z)}\cr
  & \quad \quad \quad +\underbrace{(D_{1}^{2}(z) \cdot D_2(y))\cdot D_{2}(\alpha(x))}_{g_1(x,y,z)}+\underbrace{(D_1(z)\cdot D_1 D_2(y))\cdot D_2(\alpha(x))}_{h_1(x,y,z)}\Big).
   \end{align*}
Using \eqref{Homass:homassociator}-\eqref{derivation1}, observe that
\begin{align*}
  & \circlearrowleft_{x,y,z} (b_1(x,y,z)-d_1(x,y,z))\\
  &\quad\quad=\circlearrowleft_{x,y,z}\Big( \alpha(D_1(x))\cdot(D_1(y)\cdot D_{2}^{2}(z))-\alpha(D_1(x))\cdot(D_1(z)\cdot D_{2}^{2}(y))\Big)\\
    &\quad\quad=\circlearrowleft_{x,y,z}\Big( \alpha(D_1(x))\cdot(D_1(y)\cdot D_{2}^{2}(z))-\alpha(D_1(z))\cdot(D_1(x)\cdot D_{2}^{2}(y))\Big)=0,
\\
   &\circlearrowleft_{x,y,z} (g_1(x,y,z)-e_1(x,y,z))\\&\quad\quad=\circlearrowleft_{x,y,z}\Big( (D_{1}^{2}(z) \cdot D_2(y))\cdot \alpha(D_{2}(x))-(D_{1}^{2}(y)\cdot D_2(z))\cdot \alpha(D_2(x))\Big)\\
    &\quad\quad=\circlearrowleft_{x,y,z}\Big( (D_{1}^{2}(z) \cdot D_2(y))\cdot \alpha(D_{2}(x))-(D_{1}^{2}(y)\cdot D_2(x))\cdot \alpha(D_2(z))\Big)=0,
\\
   &\circlearrowleft_{x,y,z} (a_1(x,y,z)-f_1(x,y,z))\\&\quad\quad=\circlearrowleft_{x,y,z}\Big( \alpha(D_1(x))\cdot(D_2 D_1(y)\cdot D_{2}(z))-(D_1(y)\cdot D_{1}D_2(z))\cdot \alpha(D_{2}(x))\Big)\\
    &\quad\quad=\circlearrowleft_{x,y,z}\Big( \alpha(D_1(x))\cdot(D_1 D_2(y)\cdot D_{2}(z))-\alpha(D_1(y))\cdot(D_1 D_2(z)\cdot D_{2}(x))\Big)\\&\quad\quad=0,
    \quad \mbox{(\hbox{using} $D_1D_2=D_2 D_1$)}\\
   &\circlearrowleft_{x,y,z} (h_1(x,y,z)-c_1(x,y,z))\\&\quad\quad=\circlearrowleft_{x,y,z}\Big( (D_1(z)\cdot D_1 D_2(y))\cdot \alpha(D_2(x))-\alpha(D_1(x))\cdot(D_2 D_1(z)\cdot D_2(y))\Big)\\
    &\quad\quad=\circlearrowleft_{x,y,z}\Big( (D_1(z)\cdot D_1 D_2(y))\cdot \alpha(D_2(x))-(D_1(x)\cdot D_1 D_2(z))\cdot \alpha(D_2(y))\Big)\\&\quad\quad=0. \quad \mbox{(\hbox{using}~$D_1D_2=D_2 D_1$)}
\end{align*}
This shows that $(A,\{\cdot,\cdot\},\alpha)$ satisfies the
Hom-Jacobi identity $\circlearrowleft_{x,y,z} \{\alpha(x),\{y,z\}\}=0$.

Next, for all $x,y,z\in A$,
\begin{align*}
  \{x,z\}\cdot \alpha(y)+\alpha(x)\cdot\{y,z\}&=\underbrace{(D_1(x)\cdot D_2(z))\cdot\alpha(y)}_{i_1(x,y,z)}-\underbrace{(D_1(z)\cdot D_2(x))\cdot\alpha(y)}_{j_1(x,y,z)}\\
  &\quad+\underbrace{\alpha(x)\cdot(D_1(y)\cdot D_2(z))}_{k_1(x,y,z)}-\underbrace{\alpha(x)\cdot(D_1(z)\cdot D_2(y))}_{l_{1}(x,y,z)}.
    \end{align*}
    Using \eqref{Homass:homassociator}-\eqref{derivation1},
    \begin{align*}
      &  i_1(x,y,z)+k_1(x,y,z)\\&\quad\quad=(D_1(x)\cdot D_2(z))\cdot\alpha(y)+\alpha(x)\cdot(D_1(y)\cdot D_2(z))\\
       &\quad\quad=(D_1(x)\cdot y)\cdot \alpha (D_2(z))+(x\cdot D_1(y))\cdot \alpha (D_2(z))\\
       &\quad\quad=D_1(x\cdot y)\cdot\alpha ( D_2(z))\\
       &\quad\quad=D_1(x\cdot y)\cdot D_2(\alpha (z)),
\\
       & j_1(x,y,z)+l_1(x,y,z)\\&\quad\quad=(D_1(z)\cdot D_2(x))\cdot\alpha(y)+\alpha(x)\cdot(D_1(z)\cdot D_2(y))\\
        &\quad\quad=\alpha( D_1(z))\cdot (D_2(x)\cdot y)+\alpha (D_1(z))\cdot(x\cdot D_2(y))\\
        &\quad\quad=\alpha (D_1(z))\cdot D_2(x\cdot y)\\
        &\quad\quad=D_1(\alpha (z))\cdot D_2(x\cdot y).
    \end{align*}
This means that
    \begin{align*}
   \{x,z\}\cdot \alpha(y)+\alpha(x)\cdot\{y,z\}
   &=D_1(x\cdot y)\cdot D_2(\alpha (z))-D_1(\alpha (z))\cdot D_2(x\cdot y)\\
   &=\{x\cdot y,\alpha(z)\}.
   \end{align*}
Therefore, $(A,\cdot,\{\cdot,\cdot\},\alpha)$ satisfies the
Hom-Leibniz identity.
\end{proof}
\begin{defn}
A transposed Hom-Poisson algebra is a quadruple
$(A,\cdot,\{\cdot,\cdot\},\alpha)$, where $(A,\cdot,\alpha)$ is a
commutative Hom-associative algebra and
$(A,\{\cdot,\cdot\},\alpha)$ is a Hom-Lie algebra, satisfying for all $x,y,z\in A$ the
transposed Hom-Leibniz identity,
\begin{equation}\label{leibniz}
2\alpha(z)\cdot\{x,y\}=\{z\cdot x,\alpha(y)\}+\{\alpha(x),z\cdot y\}.
\end{equation}
\end{defn}
 \begin{defn}
Let $(A,\cdot,\{\cdot,\cdot\},\alpha)$ a transposed Hom-Poisson algebra. A subspace $H$ of $A$ is called
\begin{enumerate}
\item
Hom-subalgebra of $(A,\cdot,\{\cdot,\cdot\},\alpha)$ if $\alpha(H)\subseteq H,~H\cdot H \subseteq H ~\textrm{and} ~\{H,H\}\subseteq H.$
\item
Hom-ideal of $(A,\cdot,\{\cdot,\cdot\},\alpha)$ if $\alpha(H)\subseteq H,~A\cdot H\subseteq H~\textrm{and} ~\{A,H\}\subseteq H.$
  \end{enumerate}
\end{defn}
\begin{prop}
Let  $(A_1,\cdot_1,\{\cdot,\cdot\}_1,\alpha_1)$ and $(A_2,\cdot_2,\{\cdot,\cdot\}_2,\alpha_2)$ be two transposed Hom-Poisson algebras.
Define two operations $\cdot$ and $[\cdot,\cdot]$ on $A_1\otimes A_2$ for all $x_1,y_1\in A_1$,  $x_2,y_2\in A_2$ by
\begin{align}
(x_1\otimes x_2)\cdot (y_1\otimes y_2)&=x_1\cdot_1 y_1\otimes x_2\cdot_2
y_2,  \label{eq:tensor1:HomtransposedPoissonalgebras}\\
[x_1\otimes x_2,y_1\otimes y_2] &=\{x_1,y_1\}_1\otimes x_2\cdot_2
y_2+x_1\cdot_1 y_1\otimes \{x_2,y_2\}_2,  \label{eq:tensor2:HomtransposedPoissonalgebras}\\
(\alpha_1\otimes \alpha_2) (x_1\otimes x_2) &=\alpha_1(x_1)\otimes\alpha_2(x_2), \label{eq:tensor3:HomtransposedPoissonalgebras}
\end{align}
Then, $(A_1\otimes A_2,\cdot, [\cdot,\cdot],\alpha_1\otimes\alpha_2)$ is a transposed Hom-Poisson algebra.
\end{prop}
\begin{prop}
Let $(A,\cdot,\alpha)$ be a commutative Hom-associative algebra, and let $D$ be a derivation. Define the multiplication, for all $x, y\in A$, by
\begin{equation}
\{x,y\}= x\cdot D(y)-D(x)\cdot y.
\label{eq:twoder}
\end{equation}
Then, $(A,\cdot,\{\cdot,\cdot\},\alpha)$ is a transposed Hom-Poisson algebra.
\end{prop}

\begin{proof}
First, we prove that $(A,\{\cdot,\cdot\},\alpha)$ is a Hom-Lie algebra. For all $x,y,z\in A$,
\begin{align*}
\{x,y\}&=x\cdot D(y)-D(x)\cdot y=-(y\cdot D(x)-D(y).x)=-\{y,x\}, \\
\circlearrowleft_{x,y,z} & \{\alpha(x),\{y,z\}\}=\circlearrowleft_{x,y,z}\{\alpha(x),y\cdot D(z)-D(y)\cdot z\}\\
&=\circlearrowleft_{x,y,z}\Big(\alpha(x)\cdot D(y\cdot D(z)-D(y)\cdot z)-D(\alpha (x))\cdot(y\cdot D(z)-D(y)\cdot z)\Big)\\
&= \circlearrowleft_{x,y,z}\Big(\underbrace{\alpha(x)\cdot(D(y)\cdot D(z))}_{a_2(x,y,z)}+\underbrace{\alpha(x)\cdot(y\cdot D^{2}(z))}_{b_2(x,y,z)}-\underbrace{\alpha(x)\cdot(D^{2}(y)\cdot z)}_{c_2(x,y,z)}\\
&\quad \quad \quad -\underbrace{\alpha(x)\cdot(D(y)\cdot D(z))}_{d_2(x,y,z)}-\underbrace{D(\alpha(x))\cdot(y\cdot D(z))}_{e_2(x,y,z)}+\underbrace{D(\alpha(x))\cdot(D(y)\cdot z)}_{f_2(x,y,z)}\Big).
\end{align*}
Using \eqref{Homass:homassociator}-\eqref{derivation1},
\begin{align*}
  & \circlearrowleft_{x,y,z} (a_2(x,y,z)-d_2(x,y,z))\\&\quad\quad=\circlearrowleft_{x,y,z}\Big( \alpha(x)\cdot(D(y)\cdot D(z))-\alpha(x)\cdot(D(y)\cdot D(z))\Big)=0,
\\
  & \circlearrowleft_{x,y,z} (b_2(x,y,z)-c_2(x,y,z))\\&\quad\quad=\circlearrowleft_{x,y,z}\Big( \alpha(x)\cdot(y\cdot D^{2}(z))-\alpha(x)\cdot(D^{2}(y)\cdot z)\Big)\\
    &\quad\quad=\circlearrowleft_{x,y,z}\Big( \alpha(x)\cdot(y\cdot D^{2}(z))-\alpha(z)\cdot(x\cdot D^{2}(y))\Big)=0,
\\
  & \circlearrowleft_{x,y,z} (f_2(x,y,z)-e_2(x,y,z))\\&\quad\quad=\circlearrowleft_{x,y,z}\Big( \alpha (D(x))\cdot(D(y)\cdot z)-\alpha (D(x))\cdot(y\cdot D(z))\Big)\\
    &\quad\quad=\circlearrowleft_{x,y,z}\Big( \alpha (D(x))\cdot(D(y)\cdot z)-\alpha (D(z))\cdot(D(x)\cdot y)\Big)=0.
\end{align*}
This proves that $(A,\{\cdot,\cdot\},\alpha)$ satisfies the
Hom-Jacobi identity.
Next, for all $x,y,z\in A$,
\begin{align*}
    &\{z\cdot x,\alpha(y)\}+\{\alpha(x),z\cdot y\}\\
    &\quad =(z\cdot x)\cdot  D(\alpha(y))-D(z\cdot x)\cdot \alpha(y)+\alpha(x)\cdot D(z\cdot y)- D(\alpha(x))\cdot (z\cdot y)\\
    &\quad=\underbrace{(z\cdot x)\cdot  D(\alpha(y))}_{i_2(x,y,z)}-\underbrace{(D(z)\cdot x)\cdot\alpha(y)}_{j_2(x,y,z)}-\underbrace{(z\cdot D(x))\cdot\alpha(y)}_{k_2(x,y,z)}\\
    &\quad \quad
    +\underbrace{\alpha(x)\cdot(D(z)\cdot y)}_{l_2(x,y,z)}+\underbrace{\alpha(x)\cdot(z\cdot D(y))}_{m_2(x,y,z)}-\underbrace{D(\alpha (x))\cdot (z\cdot y)}_{n_2(x,y,z)}.
\end{align*}
    Using \eqref{Homass:homassociator}-\eqref{derivation1},
    \begin{align*}
    &   l_2(x,y,z)-j_2(x,y,z)\\&\quad\quad=\alpha(x)\cdot(D(z)\cdot y)-(D(z)\cdot x)\cdot\alpha(y)\\
       &\quad\quad=\alpha(x)\cdot(D(z)\cdot y)-\alpha(x)\cdot(D(z)\cdot y)=0,\\
      &  i_2(x,y,z)+m_2(x,y,z)\\&\quad\quad=(z\cdot x)\cdot  \alpha(D(y))+\alpha(x)\cdot(z\cdot D(y))\\
        &\quad\quad=\alpha(z)\cdot(x\cdot D(y))+\alpha(z)\cdot(x\cdot D(y))=2\alpha(z)\cdot(x\cdot D(y)),
\\
     &   k_2(x,y,z)+n_2(x,y,z)\\&\quad\quad=(z\cdot D(x))\cdot\alpha(y)+\alpha (D(x))\cdot (z\cdot y)\\
        &\quad\quad=\alpha(z)\cdot(D(x)\cdot y)+\alpha(z)\cdot(D(x)\cdot y)=2\alpha(z)\cdot(D(x)\cdot y).
    \end{align*}
This means that
$$ \{z\cdot x,\alpha(y)\}+\{\alpha(x),z\cdot y\}
=2\alpha(z)\cdot(x\cdot D(y))-2\alpha(z)\cdot(D(x)\cdot y)=2\alpha(z)\{x,y\}.
$$
Therefore, $(A,\cdot,\{\cdot,\cdot\},\alpha)$ satisfies the transposed Hom-Leibniz identity.
\end{proof}
\begin{ex}\label{examplehomtransposed} Consider a two dimensional $\mathbb{K}$-linear space $A$ with basis $\{e_1,e_2\}$.
	Consider the commutative Hom-associative  algebra $(A,\cdot,\alpha)$ defined on $ A$ by $$e_1\cdot e_1= -e_1,~
		e_1\cdot e_2= e_2\cdot e_1=e_2,~ \alpha(e_1)=-e_1,~\alpha(e_2)=-e_2,$$ and the derivation $D$ on $A$ defined by $D(e_1)=0,~D(e_2)=\lambda e_2,~~\lambda\in\mathbb{K}.$
Then, $(A,\cdot,\{\cdot,\cdot\},\alpha)$ is a transposed Hom-Poisson algebra, where
for all $(i,j)\neq(1,2)$,
\begin{align}
\{e_1,e_2\}&= e_1\cdot D(e_2)-D(e_1)\cdot e_2=\lambda e_2, \\
\{e_i,e_j\}&=0.
\end{align}
\end{ex}

\begin{lem}[\cite{BaiBaiGuoWu2020:transpPoisNovPois3Lie}]\label{lemma1.1}
Let $(A,\cdot,\{\cdot,\cdot\})$ be a transposed Poisson algebra. Then, for all $x,y,z,t\in A,$
\begin{equation}\label{lemmaidentity2}
\{xz,yt\}+\{xt,yz\}=2zt\{x,y\}.
\end{equation}
\end{lem}
\begin{proof}
Using \eqref{trans Poisson},
\begin{align*}
\{xz,yt\}+\{x,ytz\}=2z\{x,yt\},\\
\{xt,yz\}+\{xtz,y\}=2z\{xt,y\}.
\end{align*}
Taking the sum of the above two identities gives
$$ (\{xz,yt\}+\{xt,yz\})+(\{x,ytz\}+\{xtz,y\})= 2z(\{x,yt\}+\{xt,y\}).$$
This means that
\begin{align*}
    \{x,ytz\}+\{xtz,y\}=2tz\{x,y\},\\
    2z(\{x,yt\}+\{xt,y\})=4zt\{x,y\}.
\end{align*}
Therefore,  \eqref{lemmaidentity2} is satisfied.
\end{proof}
\begin{lem}[\cite{BaiBaiGuoWu2020:transpPoisNovPois3Lie}]\label{lemma1.2}
Let $(A, \cdot, \{\cdot,\cdot\})$ be a transposed Poisson algebra. For any $h \in A$, define a
linear map $\alpha_h : A\rightarrow A$ by
\begin{equation}\label{alpha_h}
\alpha_h(x) = h\cdot x, ~~\hbox{for~all}~x \in A.
\end{equation}
Then, $(A, \{\cdot,\cdot\}, \alpha_h)$ is a Hom-Lie algebra.
\end{lem}
\begin{lem}\label{lemma1.3}
Let $(A, \cdot)$ be a commutative associative algebra and for all $h\in A$
$\alpha_h$ the linear map define in  \eqref{alpha_h}.
Then, $(A, \cdot,\alpha_h)$ is a commutative Hom-associative algebra.
\end{lem}
\begin{proof}
For all $x,y,z\in A$,
$$(x\cdot y)\cdot\alpha_{h}(z)=(x\cdot y)\cdot h\cdot z=h\cdot(x\cdot y)\cdot  z=h\cdot x\cdot( y\cdot  z)=\alpha_h(x)\cdot(y\cdot z).$$
This means that  $(A, \cdot,\alpha_h)$ is a commutative Hom-associative algebra.
\end{proof}
\begin{thm}\label{theorem transposed to homtransposed}
Let $(A, \cdot, \{\cdot,\cdot\})$ be a transposed Poisson algebra. For each $h\in A$, let
$\alpha_h$ be the linear map defined in  \eqref{alpha_h}.
Then, $(A, \cdot,\{\cdot,\cdot\}, \alpha_h)$ is a transposed Hom-Poisson algebra.
\end{thm}
\begin{proof}
Let $(A,\cdot,\{\cdot,\cdot\})$ be a transposed Poisson algebra. By Lemma \ref{lemma1.2} and Lemma \ref{lemma1.3},
$(A, \{\cdot,\cdot\},\alpha_h)$ is a Hom-Lie algebra and $(A,\cdot,\alpha_h)$ is a
commutative Hom-associative algebra respectively.
Now, using Lemma \ref{lemma1.1}, we show that the transposed Hom-Leibniz identity is satisfied:
\begin{align*}
 \{z\cdot x,\alpha_h(y)\}+\{\alpha_h(x),z\cdot y\}
   &= \{z\cdot x,h\cdot y\}+\{h\cdot x,z\cdot y\}\\
   &=2z\cdot h\{x,y\} \quad\quad \mbox{(by \eqref{lemmaidentity2})}\\
   &=2\alpha_h(z)\{x,y\}
\qedhere \end{align*}
\end{proof}
\begin{ex}
Consider a two-dimensional $\mathbb{K}$-linear space $A$ with basis $\{e_1,e_2\}$, and
	consider the transposed Poisson algebra $(A,\cdot,\{\cdot,\cdot\})$ defined on $ A$ by
$$e_1\cdot e_2=e_2\cdot e_1= e_1,~e_2\cdot e_2=e_2,~\{e_1,e_2\}=e_2.$$
Let $\alpha_{e_1}:A\rightarrow A$ be the  linear map defined by $\alpha_{e_1}(e_1)=0, \alpha_{e_1}(e_2)=e_1.$
Then, by Theorem \ref{theorem transposed to homtransposed}, $(A,\cdot,\{\cdot,\cdot\},\alpha_{e_1})$ is a transposed Hom-Poisson algebra.
\end{ex}
\begin{lem}
Let $(A,\cdot,\{\cdot,\cdot\},\alpha)$ be a transposed Hom-Poisson algebra.
Then, for all $x,y,z\in A$,
\begin{equation}\label{lemmaidentity}\alpha(x)\cdot\{y,z\}+\alpha(y)\cdot\{z,x\}+\alpha(z)\cdot\{x,y\}=0.\end{equation}
\end{lem}
\begin{proof}
Let $x,y,z\in A$. By \eqref{leibniz},
\begin{align*}
    \{x\cdot z,\alpha(y)\}+\{\alpha(x),y\cdot z\}=2 \alpha(z)\cdot\{x,y\},\\
      \{y\cdot x,\alpha(z)\}+\{\alpha(y),z\cdot x\}=2 \alpha(x)\cdot\{y,z\},\\
        \{z\cdot y,\alpha(x)\}+\{\alpha(z),x\cdot y\}=2 \alpha(y)\cdot\{z,x\}.
\end{align*}
Taking the sum of the three identities above gives \eqref{lemmaidentity}.
\end{proof}
\begin{prop}\label{prophompoisson}
For a linear space $A$ with a commutative Hom-associative multiplication $\cdot$ and
a Hom-Lie bracket $\{\cdot,\cdot\}$.
\begin{enumerate}
\item \label{prophompoisson:i}
If either $\cdot$ or $\{\cdot,\cdot\}$ is zero, then $(A, \cdot, \{\cdot,\cdot\},\alpha)$ is a transposed Hom-Poisson algebra, as well as
a Hom-Poisson algebra.
\item \label{prophompoisson:ii}
$(A, \cdot, \{\cdot,\cdot\},\alpha)$ is both a Hom-Poisson algebra and a transposed Hom-Poisson algebra if and
only if
$\alpha(x)\cdot\{y,z\}=\{x\cdot y,\alpha(z)\}=0$
for all $x,y,z\in A.$
\end{enumerate}
\end{prop}
\begin{proof} The statement \ref{prophompoisson:i} holds, as all axioms are trivially satisfied.
Let us prove \ref{prophompoisson:ii}. It is easy to check that if $\alpha(x)\cdot\{y,z\}=\{x\cdot y,\alpha(z)\}=0$, for all $x,y\in A$, then $(A, \cdot, \{\cdot,\cdot\},\alpha)$ is both a Hom-Poisson algebra and a transposed Hom-Poisson algebra. Conversely, let $x,y,z\in A$. By \eqref{CompatibiltyPoisson},
$$\{z\cdot x,\alpha(y)\}=-\alpha(x)\cdot\{y,z\}-\alpha(z)\cdot\{y,x\},~~\{\alpha(x),z\cdot y\}=\alpha(y)\cdot\{x,z\}+\alpha(z)\cdot\{x,y\}.$$
Next, by \eqref{leibniz},
$$0=\{z\cdot x,\alpha(y)\}+\{\alpha(x),z\cdot y\}-2\alpha(z)\cdot\{x,y\}=\{x,z\}\cdot\alpha(y)-\{y,z\}\cdot\alpha(x).$$
Then, by \eqref{lemmaidentity},
$\{x,y\}\cdot\alpha(z)=\{x,z\}\cdot\alpha(y)-\{y,z\}\cdot\alpha(x)=0.$
By \eqref{CompatibiltyPoisson} again,  $\{x\cdot y,\alpha(z)\}=0.$
\end{proof}
\begin{ex}
By Proposition \ref{prophompoisson}, the transposed Hom-Poisson algebra $(A,\cdot,\{\cdot,\cdot\},\alpha)$ in Example \ref{examplehomtransposed} is Hom-Poisson algebra if and only if $\lambda=0$.
\end{ex}
\begin{rmk}
The above Proposition means that, with the relations of the commutative
Hom-associative product and the Hom-Lie bracket, the intersection of the Hom-Leibniz identity and transposed Hom-Leibniz identity is trivial. That is, the structures
of transposed Hom-Poisson algebras and Hom-Poisson algebras are incompatible.
\end{rmk}
\begin{defn} Let $(A, \cdot, \{\cdot,\cdot\} , \alpha),$ and $(A', \cdot', \{\cdot,\cdot\}', \alpha')$ be transposed Hom-Poisson algebras.
A morphism $f:(A, \cdot, \{\cdot,\cdot\} , \alpha)\rightarrow (A', \cdot', \{\cdot,\cdot\}', \alpha')$ is a linear map $f:A\rightarrow A'$ satisfying for all $x, y\in A$,
\begin{align*}
f(x\cdot y)=f(x)\cdot ' f(y),\quad
f(\{x,y\})=\{f(x),f(y)\}',\quad
f\circ \alpha =\alpha'\circ f.
\end{align*}
\end{defn}
The following theorem provides a procedure for construction of the transposed Hom-Poisson algebras
from transposed Poisson algebras and morphisms.
\begin{thm}\label{twist-transposed}
Let $\mathcal{A}=(A, \cdot ,\{\cdot,\cdot\})$ be a transposed Poisson algebra and
$\alpha :\mathcal{A}\rightarrow \mathcal{A}$ be a transposed Poisson algebras
morphism. Define $\cdot_{\alpha}, \ \{\cdot,\cdot\}_\alpha:A \times A\rightarrow A$ for all $x, y\in A$, by
$x\cdot _{\alpha}y =\alpha (x\cdot y)$ and $\{x,y\}_{\alpha} =\alpha(\{x, y\})$.
Then, $\mathcal{A}_\alpha=(A_{\alpha}=A, \cdot _{\alpha},\{\cdot,\cdot\}_\alpha, \alpha)$ is a transposed Hom-Poisson algebra called the $\alpha$-twist or Yau twist of $(A, \cdot ,\{\cdot,\cdot\})$.
Moreover, assume that $\mathcal{A}'=(A', \cdot',\{\cdot,\cdot\}')$ is another transposed Poisson algebra and $\alpha':\mathcal{A}'\rightarrow \mathcal{A}'$ is a transposed Poisson algebras morphism. Let $f:\mathcal{A}\rightarrow \mathcal{A}'$ be a
transposed Poisson algebras morphism satisfying $f\circ \alpha =\alpha
'\circ f$. Then, $f:\mathcal{A}_{\alpha }\rightarrow \mathcal{A}'_{\alpha}$ is a
transposed Hom-Poisson algebras morphism.
\end{thm}
\begin{defn}
\label{def:deform}
Let $\mathcal{A}=(A,\cdot,\{\cdot,\cdot\})$ be a transposed Poisson algebra and $\alpha: \mathcal{A} \to \mathcal{A}$ be a transposed Poisson algebras morphism.
\begin{enumerate}
\item
$\mathcal{A}_{\alpha}=(A_{\alpha}=A,\cdot _{\alpha} = \alpha\circ \cdot,\{\cdot,\cdot\}_{\alpha} =\alpha\circ\{\cdot,\cdot\})$ is called the $\alpha$-twisting of $\mathcal{A}$.
\item
The $\alpha$-twisting $\mathcal{A}'_{\alpha}$ of $\mathcal{A}$ is called trivial if
$
\cdot_{\alpha}= 0 =\{\cdot,\cdot\}_{\alpha}$, and non-trivial if either  $\cdot_{\alpha} \neq 0$ or $\{\cdot,\cdot\}_{\alpha} \neq 0$.
\item
$\mathcal{A}$ is called rigid if every twisting of $\mathcal{A}$ is either trivial or isomorphic to $\mathcal{A}$.
\end{enumerate}
\end{defn}

\begin{prop}
\label{prop:nonrigidity}
Let $\mathcal{A}=(A,\cdot,\{\cdot,\cdot\})$ be a transposed Poisson algebra.  Suppose there exist a morphism $\alpha \colon \mathcal{A} \to \mathcal{A}$ such that either
$\cdot_{\alpha} = \alpha\circ\cdot$ is not associative, or
$\{\cdot,\cdot\}_{\alpha}=\alpha\circ\{\cdot,\cdot\}$ does not satisfy the Jacobi identity.
Then, $\mathcal{A}$ is not rigid.
\end{prop}

\begin{proof}
The $\alpha$-twisting $A'_{\alpha}$ is non-trivial, since otherwise $\cdot_{\alpha}$ would be commutative associative and $\{\cdot,\cdot\}_{\alpha}$ would satisfy the Jacobi identity. For the same reason, the $\alpha$-twisting $A'_{\alpha}$ cannot be isomorphic to $A$, since otherwise $\cdot_{\alpha}$ would be commutative associative and $\{\cdot,\cdot\}_{\alpha}$ would satisfy the Jacobi identity.
\end{proof}
\begin{defn}
Let $(A,\cdot,\{\cdot,\cdot\},\alpha)$ be a transposed Hom-Poisson algebra. A representation of $(A,\cdot,\{\cdot,\cdot\},\alpha)$ is a quadruple $(s,\rho,\beta,V)$ such that $(s,\beta,V)$ is a bimodule of the Hom-commutative-associative algebra $(A,\cdot,\alpha)$ and $(\rho,\beta,V)$ is a representation of the Hom-Lie algebra  $(A,\{\cdot,\cdot\},\alpha)$ satisfying, for all $x,y\in A,~v\in V,$
\begin{align}
&2s(\{x,y\})\beta(v)=\rho(\alpha(x))s(y)v-\rho(\alpha(y))s(x)v,\label{ismail.1}\\
&2s(\alpha(x))\rho(y)v=\rho(x\cdot y)\beta(v)+\rho(\alpha(y))s(x)v,\label{ismail.2}\\
&\beta(s(x)v)=s(\alpha(x))\beta(v),~~~~\beta(\rho(x)v)=\rho(\alpha(x))\beta(v).\label{ismail.0}
\end{align}
\end{defn}
\begin{prop}\label{pro2}
Let $(A,\cdot,\{\cdot,\cdot\},\alpha)$ be a transposed Hom-Poisson algebra, and let $(s,\rho,\beta,V)$ be its  representation. Then, $(A\oplus V,\cdot',\{\cdot,\cdot\}',\alpha+\beta)$ is a transposed Hom-Poisson algebra, where $(A\oplus V,\cdot',\alpha+\beta)$ is the semi-direct product commutative Hom-associative algebra $A\ltimes_{s,\alpha,\beta} V$, and $(A\oplus V,\{\cdot,\cdot\}',\alpha+\beta)$ is the semi-direct product Hom-Lie algebra $A\ltimes_{\rho,\alpha,\beta} V$.
\end{prop}
\begin{proof}
Let $(A,\cdot,\{\cdot,\cdot\},\alpha)$ be a transposed Hom-Poisson algebra, and let $(s,\rho,\beta,V)$ be its  representation. By Proposition \ref{ass1} and Proposition \ref{pro11},
$(A\oplus V,\cdot',\alpha+\beta)$ is a commutative Hom-associative algebra, and
$(A\oplus V,\{\cdot,\cdot\}',\alpha+\beta)$ is a Hom-Lie algebra respectively.
Now, we show that the transposed Hom-Leibniz identity is satisfied.
For all $x_1,x_2,x_3\in A$ and $v_1,v_2,v_3\in V,$
\begin{align*}
&2(\alpha+\beta)(x_3+v_3)\cdot'\{x_1+v_1,x_2+v_2\}' -\{(x_3+v_3)\cdot'(x_1+v_1),(\alpha+\beta)(x_2+v_2)\}'\\
&\quad \quad -\{(\alpha+\beta)(x_1+v_1),(x_3+v_3)\cdot'(x_2+v_2)\}'\\
&=2(\alpha+\beta)(x_3+v_3)\cdot'(\{x_1,x_2\}+\rho(x_1)v_2-\rho(x_2)v_1)\\
&\quad \quad -\{x_3\cdot x_1+s(x_3)v_1+s(x_1)v_3,(\alpha+\beta)(x_2+v_2)\}'\\
&\quad \quad -\{(\alpha+\beta)(x_1+v_1),x_3\cdot x_2+s(x_3)v_2+s(x_2)v_3\}'\\
&=2\Big(\alpha(x_3)\cdot\{x_1,x_2\}+s(\alpha(x_3))\rho(x_1)v_2-s(\alpha(x_3))\rho(x_2)v_1+s((\{x_1,x_2\})\beta(v_3)\Big)\\
&\quad \quad -\Big(\{x_3\cdot x_1,\alpha(x_2)\}+\rho(x_3\cdot x_1)\beta(v_2) -\rho(\alpha(x_2))s(x_3)v_1-\rho(\alpha(x_2))s(x_1)v_3\Big)\\
&\quad \quad -\Big(\{\alpha(x_1),x_3\cdot x_2\} +\rho(\alpha(x_1))s(x_3)v_2+\rho(\alpha(x_1))s(x_2)v_3-\rho(x_3\cdot x_2)\beta(v_1)\Big)\\
&=\Big(2\alpha(x_3)\cdot\{x_1,x_2\}-\{x_3\cdot x_1,\alpha(x_2)\}-\{\alpha(x_1),x_3\cdot x_2\}\Big)\\
&\quad \quad +\Big(2s(\alpha(x_3))\rho(x_1)v_2-\rho(x_3\cdot x_1)\beta(v_2)-\rho(\alpha(x_1))s(x_3)v_2\Big)\\
&\quad \quad -\Big(2s(\alpha(x_3))\rho(x_2)v_1-\rho(x_3\cdot x_2)\beta(v_1)-\rho(\alpha(x_2))s(x_3)v_1\Big)\\
&\quad \quad +\Big(2s(\{x_1,x_2\})\beta(v_3)+\rho(\alpha(x_2)s(x_1)v_3-\rho(\alpha(x_1))s(x_2)v_3\Big) =0+0+0+0=0.
\end{align*}
Then, $(A\oplus V,\cdot',\{\cdot,\cdot\}',\alpha+\beta)$ is a transposed Hom-Poisson algebra.
\end{proof}
\noindent
The transposed Hom-Poisson algebra $(A\oplus V,\cdot',\{\cdot,\cdot\}',\alpha+\beta)$ in
Proposition \ref{pro2} is denoted by $A \times_{s,\rho, \alpha,\beta} V.$

\begin{ex}
Let $(A,\cdot,\{\cdot,\cdot\},\alpha)$ be a transposed Hom-Poisson algebra, and let $S_{\cdot}(a)b=a\cdot b=b\cdot a$ and $ad(a)b=\{a,b\}$, for all $a,b\in A$.
Then, $(S_{\cdot},ad,\alpha,A)$ is a regular representation of $A$.
\end{ex}
\begin{thm}
Let $(A,\cdot,\{\cdot,\cdot\},\alpha)$ be a transposed Hom-Poisson algebra and $(s,\rho,\beta,V)$ a representation. Define $s^{\ast},\rho^{\ast}:A\rightarrow End(V^{\ast})$ by
\begin{equation}\label{eqq1}\langle s^{\ast}(x)u^{\ast},v\rangle=-\langle s(x)v,u^{\ast}\rangle,~~\langle \rho^{\ast}(x)u^{\ast},v\rangle=-\langle \rho(x)v,u^{\ast}\rangle.\end{equation}
Let $\alpha^{\ast}:A^{\ast}\rightarrow A^{\ast},~\beta^{\ast}:V^{\ast}\rightarrow V^{\ast}$
be the dual maps of $\alpha$ and $\beta$ respectively, that is,
\begin{equation}\label{eqq2}
\alpha^{\ast}(x^{\ast})(y)=x^{\ast}(\alpha(y)),~~\beta^{\ast}(u^{\ast})(v)=u^{\ast}(\beta(v)).
\end{equation}
If, in addition, for any $x,y\in A,~x^{\ast}\in A^{\ast},~v\in V,~v^{\ast}\in V^{\ast},$
\begin{align}
&2s(\{x,y\})\beta(v)=s(y)\rho(\alpha(x))v-s(x)\rho(\alpha(y))v,\label{isma1l.2}\\
&2\rho(y)s(\alpha(x))v=\rho(x\cdot y)\beta(v)+s(x)\rho(\alpha(y))v,\label{ismaill.3}\\
&\beta(s(x)v)=s(x)\beta(v),~~\beta(\rho(\alpha(x))v)=\rho(x)\beta(v),\label{isma1ll.1}
\end{align}
then $(-s^{\ast},\rho^{\ast},\beta^{\ast},V^{\ast})$ is a representation of $A$.
Moreover, $A\times_{-s^{\ast},\rho^{\ast},\alpha^{\ast},\beta^{\ast}}V^{\ast}$
is also transposed Hom-Poisson algebra.
\end{thm}
\begin{proof}
For any $x,y\in A,~x^{\ast}\in A^{\ast},~v\in V,~u^{\ast}\in V^{\ast}$, according to \eqref{eqq1} and \eqref{eqq2},
\begin{align*}
    \langle\beta^{\ast}(\rho^{\ast}(x)u^{\ast}),v\rangle&=-\langle\rho^{\ast}(x)u^{\ast},\beta(v)\rangle
    =-\langle \rho(x)\beta(v),u^{\ast}\rangle\\
    &=-\langle\beta(\rho(\alpha(x))v),u^{\ast}\rangle
    =-\langle \rho(\alpha(x))v,\beta^{\ast}(u^{\ast})\rangle
    =\langle\rho^{\ast}(\alpha(x))\beta^{\ast}(u^{\ast}),v\rangle.
\end{align*}
So, \eqref{ismail.0} holds for $\rho^{\ast}$. Similarly, \eqref{ismail.0} holds for $-s^{\ast}$. According to \eqref{eqq1}-\eqref{ismaill.3},
\begin{align*}
    &\langle -2s^{\ast}(\{x,y\})\beta^{\ast}(u^{\ast})+\rho^{\ast}(\alpha(x))s^{\ast}(y)u^{\ast}+\rho^{\ast}(\alpha(y))s^{\ast}(x)u^{\ast},v\rangle\\
    &\quad \quad =\langle -2\beta(s(\{x,y\})v)+s(y)\rho(\alpha(x))v-s(x)\rho(\alpha(y))v,u^{\ast}\rangle=0,
\\
    &\langle -2s^{\ast}(\alpha(x))\rho^{\ast}(y)v-\rho(x\cdot y)\beta^{\ast}(u^{\ast})+\rho^{\ast}(\alpha(y))s^{\ast}(x)u^{\ast},v\rangle\\
    &\quad \quad =\langle \beta(\rho(x\cdot y)v)+s(x)\rho(\alpha(y))v-2\rho(y)s(\alpha(x)),u^{\ast}\rangle=0.
\end{align*}
Therefore, \eqref{ismail.1} and \eqref{ismail.2} hold for $(-s^{\ast},\rho^{\ast},\beta^{\ast},V^{\ast})$. Then, $(-s^{\ast},\rho^{\ast},\beta^{\ast},V^{\ast})$ is a representation of $(A, \cdot,\{\cdot,\cdot\},\alpha)$.
\end{proof}
\begin{ex}
  Let $(A,\cdot,\{\cdot,\cdot\},\alpha)$ be a transposed Hom-Poisson algebra. The quadruples $(A,S,ad,\alpha)$ and $(A^{\ast},-S^{\ast},ad^{\ast},\alpha^{\ast})$ are representations of transposed Hom-Poisson algebra $(A, \cdot,\{\cdot,\cdot\},\alpha)$.
\end{ex}
\begin{thm}
Let $\mathcal{A}=(A,\cdot_A,\{\cdot,\cdot\}_A,\alpha)$ and $\mathcal{B}=(B,\cdot_B,\{\cdot,\cdot\}_{B},\beta)$ be two transposed Hom-Poisson algebras. Suppose that there are linear maps $s_A,\rho_A:A\rightarrow gl(B)$
and $s_B,\rho_B:B\rightarrow gl(A)$ such that $A\bowtie^{\rho_A,\beta}_{\rho_B,\alpha}B$ is a matched pair of Hom-Lie algebras, and  $A\bowtie^{s_A,\beta}_{s_B,\alpha}B$ is a matched pair of commutative Hom-associative algebras, and for all $x,y\in A,~a,b\in B$, the following equalities hold:
\begin{align}\label{101}
\begin{split}
2s_A(\alpha(x))\{a,b\}_B &=\rho_B(s_B(a)x)\beta(b)+\{s_A(x)a,\beta(b)\}_B \\
 &\quad -\rho_A(s_B(b)x)\beta(a)+\{\beta(a),s_A(x)b\}_B,
\end{split} \\
\label{102}
\begin{split}
2\beta(a)\cdot_B\rho_A(x)b &-2s_A(\rho_B(b)x)\beta(a) \\ &=\{s_A(x)a,\beta(b)\}_B+\rho_A(s_B(a)x)\beta(b)+\rho_A(\alpha(x))(a\cdot_B b),
\end{split} \\
\begin{split}\label{103}
2s_B(\beta(a))\{x,y\}_A & =\rho_A(s_A(x)a)\alpha(a)+\{s_B(a)x,\alpha(y)\}_A \\ &
\quad -\rho_B(s_A(y)a)\alpha(x)+\{\alpha(y),s_B(a)y\}_A,
\end{split} \\
\label{104}
\begin{split}
2\alpha(x)\cdot_A\rho_B(a)y & -2s_B(\rho_A(y)a)\alpha(x)\\
& =\{s_B(a)x,\alpha(y)\}_A+\rho_B(s_A(x)a)\alpha(y)+\rho_B(\beta(a))(x\cdot_A y).
\end{split}
\end{align}
Then, $(A,B,s_A,\rho_A,\beta,s_B,\rho_B,\alpha)$ is called a matched pair of transposed Hom-Poisson algebras. In this case, on the direct sum
$A\oplus B$ of the underlying vector spaces of $\mathcal{A}$ and $\mathcal{B}$, there is a transposed Hom-Poisson algebra structure which is given for any $x,y\in A,~a,b\in B$ by
\begin{align*}
(x + a) \cdot (y + b)&=x \cdot_A y + (l_A(x)b + r_A(y)a)+a \cdot_B b + (l_B(a)y + r_B(b)x), \\
[x+a,y+b]&=\{x,y\}_A+\rho_A(x)b-\rho_A(y)a+\{a,b\}_B+\rho_B(a)y-\rho_B(b)x, \\
(\alpha\oplus\beta)(x + a)&=\alpha(x) + \beta(a).
\end{align*}
\end{thm}
\begin{proof}
By Theorem \ref{matched ass} and Theorem \ref{matched Lie}, we deduce that $(A\oplus B,\cdot,\alpha+\beta)$
is a commutative Hom-associative algebra and $(A\oplus B,[\cdot,\cdot],\alpha+\beta)$ is a Hom-Lie algebra.
Now, the
rest, it is easy (in a similar way as for Proposition \ref{matched ass}) to verify the transposed Hom-Leibniz identity satisfied.
\end{proof}
\section{Manin triples and bialgebras of transposed Hom-Poisson algebras}\label{sec3}
In this section, we introduce the notion of Manin triples for  transposed Hom-Poisson algebras and transposed Hom-Poisson bialgebras and establish the equivalences between their
Manin triples.

\begin{defn}
 A bilinear form $\mathfrak B$
 on a transposed Hom-Poisson algebra $(A,\cdot,\{\cdot,\cdot\},\alpha)$ is called
 invariant if for all $x,y,z\in A$,
\begin{equation*}
\mathfrak{B}(x\cdot y,\alpha(z))=\mathfrak{B}(\alpha(x), y\cdot z),\quad \mathfrak{B}(\{x, y\},\alpha(z))=\mathfrak{B}(\alpha(x), \{y, z\}).
\end{equation*}
\end{defn}
\begin{defn}
A Manin triple of transposed Hom-Poisson algebras is a triple of transposed Hom-Poisson algebras $(A,A_1,A_2)$ together with a nondegenerate symmetric invariant bilinear
form $\mathfrak{B}$ on $A$ satisfying the following conditions:
\begin{enumerate}
    \item
    $A_1$ and $A_2$ are transposed Hom-Poisson subalgebras of $A$;
    \item
    as linear spaces, $A=A_1\oplus A_2$;
    \item
$A_{1}$ and $A_{2}$ are isotropic with
respect to $\mathfrak{B}$, that is,
for any $x_1,y_1\in A_1$ and any $x_2, y_2\in A_2$,
$\mathfrak{B}(x_1,y_1)=0=\mathfrak{B}(x_2,y_2).$
\end{enumerate}
\end{defn}
\begin{defn} Let  $(A,\cdot,\{\cdot,\cdot\},\alpha)$ be a transposed Hom-Poisson algebra. Suppose there is a transposed Hom-Poisson algebra structure $(A^*,\circ,[\cdot,\cdot],\alpha^{\ast})$ on the dual space of $A$, and there is a transposed Hom-Poisson algebra structure on the direct sum $A\oplus A^*$ of
the underlying linear spaces of $A$ and $A^*$, such that
$(A,\cdot,\{\cdot,\cdot\},\alpha)$ and $(A^*,\circ,[\cdot,\cdot],\alpha^{\ast})$ are subalgebras, and the natural symmetric bilinear form on $A\oplus A^*$ given, for all $x,y\in A, a^*,b^*\in A^*$, by
\begin{align}\label{eq:sbl1}
\mathfrak{B}_d(x+a^*,y+b^*) &= \langle  a^*,y\rangle +\langle x,b^*\rangle ,\; \\
\label{eq:sbl2}
(\alpha+\alpha^{\ast})(x+a^{\ast}) &= \alpha(x)+\alpha^{\ast}(a^{\ast}),
\end{align}
is
invariant, then $(A\oplus A^*,A,A^*,\alpha+\alpha^{\ast})$ is called a
standard Manin triple of transposed Hom-Poisson algebras associated to
$\mathfrak{B}_d$.
\end{defn}
\begin{rmk}
A standard Manin triple of transposed Hom-Poisson algebras is a Manin triple
of transposed Hom-Poisson algebras.
\end{rmk}
\begin{thm}
Let $(A,\cdot,\{\cdot,\cdot\},\alpha)$ and $(A^{\ast},\circ,[\cdot,\cdot],\alpha^{\ast})$ be two transposed Hom-Poisson algebras. Then, $(A\oplus A^{\ast},A,A^{\ast})$ is a standard Manin triple of transposed Poisson algebras if and only if $(A,A^{\ast},-S^{\ast}_{\cdot},ad^{\ast}_{\{\cdot,\cdot\}},\alpha^{\ast},-S^{\ast}_{\circ},ad^{\ast}_{[\cdot,\cdot]},\alpha)$ is a matched pair of transposed Hom-Poisson algebras.
\end{thm}
\begin{proof}
Consider the following four maps, given for all $ x, v, u \in A $, $ x^{\ast}, v^{\ast}, u^{\ast} \in A^{\ast} $, by
\begin{eqnarray*}
&&S^{\ast}_{\cdot}: A \rightarrow gl(A^{\ast}), \langle S^{\ast}_{\cdot}(x)u^{\ast}, v \rangle = -\langle
u^{\ast},S_{\cdot}(x)v  \rangle,\cr
&&ad^{\ast}_{\{\cdot,\cdot\}} : A \rightarrow gl(A^{\ast}), \langle ad^{\ast}_{\{\cdot,\cdot\}}(x)u^{\ast}, v \rangle = \langle
u^{\ast},ad_{\{\cdot,\cdot\}}(x)v  \rangle ,\cr
&&S^{\ast}_{\circ} : A^{\ast} \rightarrow gl(A), \langle S^{\ast}_{\circ}(x^{\ast})u,
v^{\ast} \rangle =- \langle u,S_{\circ}(x^{\ast})v^{\ast}  \rangle, \cr
&&ad^{\ast}_{[\cdot,\cdot]} : A^{\ast} \rightarrow gl(A), \langle ad^{\ast}_{[\cdot,\cdot]}(x^{\ast})u,
 v^{\ast} \rangle = \langle u , ad_{[\cdot,\cdot]}(x^{\ast})v^{\ast}\rangle.
  \end{eqnarray*}
If $ (A, A^{\ast}, -S^{\ast}_{\cdot}, ad^{\ast}_{\{\cdot,\cdot\}}, \alpha^{\ast}, -S^{\ast}_{\circ},  ad^{\ast}_{[\cdot,\cdot]}, \alpha)$ is a matched pair of transposed Hom-Poisson algebras, then $A\bowtie^{-S^{\ast}_{\cdot},ad^{\ast}_{\{\cdot,\cdot\}},\alpha^{\ast}}_{-S^{\ast}_{\circ},ad^{\ast}_{[\cdot,\cdot]},\alpha} A^{\ast}$ is a transposed Hom-Poisson algebra with the product $ \ast $ and $[\cdot,\cdot]_\rho$ given by
\begin{align}
&(x + a) \ast (y + b)=x \cdot y - (S^{\ast}_{\cdot}(x)b + S^{\ast}_{\cdot}(y)a)+a \circ b - (S^{\ast}_{\circ}(a)y + S^{\ast}_{\circ}(b)x),\cr
&[[x+a,y+b]]=\{x,y\}+ad^{\ast}_{\{\cdot,\cdot\}}(x)b
-ad^{\ast}_{\{\cdot,\cdot\}}(y)a+[a,b]+ad^{\ast}_{[\cdot,\cdot]}(a)y-ad^{\ast}_{[\cdot,\cdot]}(b)x.\nonumber
\end{align}
Let us compute and compare
$\displaystyle \mathfrak{B}_d((x+a)\ast(y+b), (z+c))$
and $\mathfrak{B}_d((x+a), (y+b)\ast(z+c))$,  for all
$x, y, z \in A, a, b, c \in A^{*}:$
\begin{align*}
&\mathfrak{B}_d((x+a)\ast(y+b),z+c)\\
&\quad =B_d(x\cdot y-S^{\ast}_{\cdot}(x)b-S^{\ast}_{\cdot}(y)a+a\circ b-S^{\ast}_{\circ}(a)y-S^{\ast}_{circ}(b)x,z+c)\\
&\quad =\langle x\cdot y-S^{\ast}_{\circ}(a)y-S^{\ast}_{\circ}(b)x,c\rangle+\langle z,a\circ b-S^{\ast}_{\cdot}(x)b-S^{\ast}_{\cdot}(y)a\rangle\\
&\quad =\langle x\cdot y,c\rangle-\langle S^{\ast}_{\circ}(a)y,c\rangle-\langle S^{\ast}_{\circ}(b)x,c\rangle+\langle z,a\circ b\rangle -\langle z, S^{\ast}_{\cdot}(x)b\rangle-\langle z,S^{\ast}_{\cdot}(y) a\rangle\\
&\quad =\langle x\cdot y,c\rangle+\langle y,S_{\circ}(a)c\rangle+\langle x,S_{\circ}(b)c\rangle+\langle z,a\circ b\rangle +\langle S_{\cdot}(x)z, b\rangle+\langle S_{\cdot}(y)z, a\rangle\\
&\quad =\langle x\cdot y,c\rangle+\langle y,a\circ c\rangle+\langle x,b\circ c\rangle+\langle z,a\circ b\rangle +\langle x\cdot z, b\rangle+\langle y\cdot z, a\rangle,\\
&\mathfrak{B}_d((x+a),(y+b)\ast(z+c))\\
&\quad =B_d((x+a),y\cdot z-S^{\ast}_{\cdot}(y)c-S^{\ast}_{\cdot}(z)b-S^{\ast}_{\circ}(b)z-S^{\ast}_{\circ}(c)y)\\
&\quad =\langle x,b\circ c-S^{\ast}_{\cdot}(y)c-S^{\ast}_{\cdot}(z)b\rangle+\langle y\cdot z-S^{\ast}_{\circ}(b)z-S^{\ast}_{\circ}(c)y,a\rangle\\
&\quad =\langle x,b\circ c\rangle-\langle x ,S^{\ast}_{\cdot}(y)c\rangle-\langle x,S^{\ast}_{\cdot}(z)b\rangle
+\langle y\cdot z,a\rangle-\langle S^{\ast}_{\circ}(b)z,a\rangle-\langle S^{\ast}_{\circ}(c)y,a\rangle\\
&\quad =\langle x,b\circ c\rangle+\langle S_{\cdot}(y)x ,c\rangle+\langle S_{\cdot}(z)x,b\rangle
+\langle y\cdot z,a\rangle+\langle z,S_{\circ}(b)a\rangle+\langle y,S_{\circ}(c)a\rangle\\
&\quad =\langle x,b\circ c\rangle+\langle y\cdot x ,c\rangle+\langle z\cdot x,b\rangle
+\langle y\cdot z,a\rangle+\langle z,b\circ a\rangle+\langle y,c\circ a\rangle.
\end{align*}
It follows that
$\mathfrak{B}_d((x+a)\ast(y+b),z+c)=\mathfrak{B}_d((x+a),(y+b)\ast(z+c)).$
Similarly,
$\mathfrak{B}_d([[x+a,y+b]],z+c)=\mathfrak{B}_d((x+a),[[y+b,z+c]]).$
We deduce the invariance of the standard bilinear form $\mathfrak{B}_d$ on $A\oplus A^{\ast}$. Therefore, $(A\oplus A^{\ast},A,A^{\ast})$ is a standard Manin triple of transposed Poisson algebras $A$ and $A^{\ast}.$
\end{proof}
\begin{defn}
A Lie bialgebra is a quadruple $(A,\{\cdot,\cdot\},\delta,\alpha)$ where $(A,\{\cdot,\cdot\},\alpha)$ is a Hom-Lie algebra and $\delta:A\rightarrow \wedge^{2} A$ is a linear map such that $(A,\delta,\alpha)$ is a Hom-Lie coalgebra and $\delta$ is a $1$-cocycle on $A$ with coefficients in $A\otimes A$, that is, it satisfies the following equation for all $x,y\in A$,
\begin{equation}
    \delta(\{x,y\})=(ad_{x}\otimes \alpha+\alpha\otimes ad_{x})\delta(y)
    -(ad_{y}\otimes \alpha+\alpha\otimes ad_{y})\delta(x).
\end{equation}
A Hom-Lie bialgebra $(A, \{\cdot,\cdot\}, \delta, \alpha)$
is called coboundary if $\delta$ is a $1$-coboundary, that is, there exists an $r \in A\otimes A$ with $\alpha^{\otimes 2}(r) = r$,
such that for all $x\in A$,
\begin{equation}
\delta(\{x,y\})=(ad_{x}\otimes \alpha+\alpha\otimes ad_{x})\delta(y))r.\end{equation}
In this case, a coboundary Lie bialgebra is usually denoted by $(A,\{\cdot,\cdot\},r,\alpha).$
\end{defn}
\begin{defn}
Let $(A, \cdot,  \alpha)$ be a  commutative Hom-associative algebra. An infinitesimal Hom-bialgebra structure on $A$ is a linear
map $\Delta: A\rightarrow A\otimes A$ such that $(A, \Delta)$ is a Hom-coassociative coalgebra and $\Delta$ is a 1-cocycle,
that is, for all $a, b\in A$,
\begin{eqnarray}
\Delta(a\cdot b)=(S(\alpha(a))\otimes \alpha)\Delta(b)+(\alpha\otimes S(\alpha(b)))\Delta(a).
\end{eqnarray}
An infinitesimal Hom-bialgebra $(A, \cdot, \Delta, \alpha)$ is called coboundary if $\Delta$
1-coboundary, that is, there exists an $r\in A\otimes A$ with $\alpha^{\otimes 2}(r)=r$, such that for all $a\in A$,
\begin{eqnarray}
\Delta(a)=(S(a)\otimes \alpha-\alpha\otimes S(a))r.
\end{eqnarray}
In this case, a coboundary infinitesimal Hom-bialgebra is denoted by $(A,\cdot,r,\alpha)$.
\end{defn}
\begin{defn}
Let $(A,\cdot, \{\cdot, \cdot\},  \alpha)$ be a transposed Hom-Poisson algebra. Suppose that it is equipped with two
comultiplications $\delta, \Delta: A\rightarrow A\otimes A$ such that $(A, \Delta,\delta,  \alpha)$ is a transposed Hom-Poisson coalgebra, that is, $(A, \delta,\alpha)$ is a
Hom-Lie coalgebra and $(A, \Delta,\alpha)$ is a Hom-cocommutative coassociative coalgebra, and they satisfy the following compatibility condition for all $x\in A$:
\begin{eqnarray*}
(\alpha\otimes \Delta)\delta(x)=(\delta\otimes \alpha)\Delta(x)+(\tau\otimes id)(\alpha\otimes \delta)\Delta(x).
\end{eqnarray*}
If in addition, $(A, \{\cdot,\cdot \}, \delta, \alpha)$ is a Hom-Lie bialgebra, $(A, \cdot, \Delta, \alpha)$ is a commutative and associative
infinitesimal Hom-bialgebra, and $\delta$ and $\Delta$ are compatible in the sense that
\begin{eqnarray*}
\delta(x\circ y)&=&(S(\alpha(y))\otimes \alpha)\delta(x)+(S(\alpha(x))\otimes \alpha)\delta(y)\\
&&-(\alpha\otimes ad_{x})\Delta(y)-(\alpha\otimes ad_{y})\Delta(x),\\
\Delta(\{x, y\})&=&(ad_{\alpha(x)}\otimes \alpha +\alpha\otimes ad_{\alpha(x)} \Delta(y)\\
&&+(S(\alpha(y))\otimes \alpha-\alpha \otimes S(\alpha(y))  )\delta(x),
\end{eqnarray*}
then $(A, \cdot,\{\cdot,\cdot \},  \delta, \Delta, \alpha)$ is called a transposed Hom-Poisson bialgebra.
\end{defn}
\begin{thm} Let $(A, \cdot,\{\cdot, \cdot\},\alpha)$ be a transposed Hom-Poisson algebra equipped with two comultiplications
$\delta, \Delta: A\rightarrow A\otimes A$. Suppose that the mappings $\delta^{\ast}, \Delta^{\ast}: A^{\ast}\otimes A^{\ast}\subset (A\otimes A)^{\ast}\rightarrow A^{\ast}$ induce a transposed Hom-Poisson algebra structure on $(A^{\ast}, \alpha^{\ast})$, where $\delta^{\ast}$ and $\Delta^{\ast}$ correspond to the Hom-Lie bracket and the product of the commutative Hom-associative algebra respectively. Set $[\cdot, \cdot]=\delta^{\ast}$, $\circ=\Delta^{\ast}$. Then, the following conditions
are equivalent:
\begin{enumerate}
\item
$(A, \{\cdot, \cdot\}, \cdot, \delta, \Delta, \alpha)$ is a transposed Hom-Poisson bialgebra.
\item
$(A, A^{\ast}, ad^{\ast}_{\{\cdot, \cdot\}}, -S^{\ast}_{\cdot},\alpha^{\ast} ad^{\ast}_{[\cdot, \cdot]}, -S^{\ast}_{\circ},\alpha)$ is a matched pair of the transposed Hom-Poisson algebras.
\item
$(A\oplus A^{\ast}, A, A^{\ast}, \alpha+\alpha^{\ast})$  is a standard Manin triple of the transposed Hom-Poisson algebras with the bilinear form defined
in \eqref{eq:sbl1}, and the isotropic subalgebras are $A$ and $A^{\ast}$.
\end{enumerate}
\end{thm}
\section{Hom-pre-Lie Poisson algebras}\label{sec4}
In this section, we provide definition of Hom-pre-Lie Poisson algebra, develop
some constructions theorems and define their bimodule and matched pair and describe some relevant properties.

\begin{defn}[\cite{ms:homstructure}]
A Hom-pre-Lie algebra is a triple $(A,\ast, \alpha)$ consisting of a vector space $A$, a bilinear product $\ast:A\otimes A\rightarrow A,$ and a linear map
$\alpha:A\rightarrow A$ satisfying, for all $x, y, z\in  A$,
\begin{equation}\label{klkl}
(x\ast y)\ast\alpha(z)-\alpha(x)\ast(y\ast z)=(y\ast
x)\ast\alpha(z)-\alpha(y)\ast(x\ast z).
\end{equation}
The equation \eqref{klkl} is called Hom-pre-Lie identity.
\end{defn}
\begin{lem}\label{lem1}
Let $(A,\ast,\alpha)$ be a Hom-pre-Lie algebra. Then,
$(A,[\cdot,\cdot],\alpha)$ is a Hom-Lie algebra with
$[x, y]=x\ast y-y\ast x,$
for any $x,y \in A$. We say that $(A,[\cdot,\cdot],\alpha)$ is the
sub-adjacent Lie algebra of $(A,\ast,\alpha)$ and denoted by $A^{c}$.
\end{lem}
\begin{proof}
First, we prove $[\cdot,\cdot]$ satisfies Hom-skewsymmetry,
$$[x,y]=x\ast y-y\ast x=-(y\ast x-x\ast y)=-[y,x].$$
Now, we prove the Hom-Jacobi-identity,
\begin{align*}
&[\alpha(x),[y,z]]+[\alpha(y),[z,x]]+[\alpha(z),[x,y]]\\
&\quad =[\alpha(x),y\ast z-z\ast y]+[\alpha(y),z\ast x-x\ast
z]+[\alpha(z),x\ast y-y\ast x]\\
&\quad =\alpha(x)\ast(y\ast z)-\alpha(x)\ast(z\ast y)-(y\ast z)\ast \alpha(x)\\
&\quad\quad +(z\ast y)\ast\alpha(x)+\alpha(y)\ast(z\ast x)-\alpha(y)\ast(x\ast z)\\
&\quad\quad -(z\ast x)\ast\alpha(y)+(x\ast z)\ast\alpha(y)+\alpha(z)\ast(x\ast y)\\
&\quad\quad -\alpha(z)\ast(y\ast x)-(x\ast y)\ast\alpha(z)+(y\ast x)\ast\alpha(z)\\
&\quad=(y\ast x)\ast\alpha(z)-\alpha(y)\ast(x\ast z)-((x\ast y)\ast\alpha(z)-\alpha(x)\ast(y\ast z))\\
&\quad\quad +(z\ast y)\ast\alpha(x)-\alpha(z)\ast(y\ast x)
-((y\ast z)\ast\alpha(x)-\alpha(z)\ast(y\ast x))\\
&\quad\quad +(x\ast z)\ast\alpha(y)-\alpha(x)\ast(z\ast y)
-((z\ast x)\ast\alpha(y)-\alpha(z)\ast(x\ast y))\\
&\quad =0+0+0=0.
\qedhere
\end{align*}
\end{proof}
Let us recall now the notion of bimodule and matched pair of a Hom-pre-Lie algebra.
\begin{defn}[\cite{SunLi2017:parakahlerhomhomleftsymetric}]
Let $(A, \ast, \alpha)$ be a Hom-pre-Lie algebra, and $(V, \beta)$ be a Hom-module. Let $ l_\ast, r_\ast: A \rightarrow gl(V) $ be two linear maps. The quadruple $(l_\ast, r_\ast, \beta, V)$ is called a bimodule of $A$ if for all $ x, y \in  A, v \in V $,
\begin{eqnarray}
l_\ast(\{x,y\})\beta(v)&=&
l_\ast(\alpha(x))l_\ast(y)v-l_\ast(\alpha(y))l_\ast(x)v\\
r_\ast(\alpha(y))\rho(x)v&=&l_\ast(\alpha(x))r_\ast(y)-r_\ast(x\ast y)\beta(v)\\
\beta(l_\ast(x)v)&=& l_\ast(\alpha(x))\beta(v),\\ \beta(r_\ast(x)v)&=& r_\ast(\alpha(x))\beta(v),
\end{eqnarray}
where $\{x,y\}=x\ast y-y\ast x$ and $\rho=l_\ast-r_\ast.$
\end{defn}
\begin{prop}[\cite{SunLi2017:parakahlerhomhomleftsymetric}] \label{bimodpre-Lie}
Let $(l_\ast, r_\ast, \beta, V)$ be a bimodule of a Hom-pre-Lie algebra
$(A, \ast, \alpha)$. Then, the direct sum $A \oplus V$ of vector spaces is turned into a
Hom-pre-Lie algebra  by defining multiplication in $A \oplus V $ for all $ x_{1}, x_{2} \in  A, v_{1}, v_{2} \in V$ by
\begin{eqnarray*}
(x_{1} + v_{1}) \ast' (x_{2} + v_{2}) & = & x_{1} \ast x_{2} + (l_\ast(x_{1})v_{2} + r_\ast(x_{2})v_{1}),\cr
(\alpha\oplus\beta)(x_{1} + v_{1}) & = & \alpha(x_{1}) + \beta(v_{1}).
\end{eqnarray*}
\end{prop}
We denote such a Hom-pre-Lie algebra by $(A\oplus V, \ast', \alpha+ \beta),$
or $A\times_{l_\ast, r_\ast, \alpha, \beta} V.$

\begin{prop}\label{propa1}
Let $( l_{\ast},r_{\ast}, \beta, V)$ be a bimodule of a  Hom-pre-Lie algebra
$(A,\ast, \alpha)$. Let $(A, \{\cdot,\cdot\}, \alpha)$ be the
subadjacent Hom-Lie algebra of $(A, \ast, \alpha)$. Then, $(l_\ast-r_\ast,\beta,V)$ is a
representation of Hom-Lie algebra $(A,\{\cdot,\cdot\},\alpha)$.
\end{prop}
\begin{proof}
For all $x,y\in A,~v\in V$,
\begin{align*}
\beta((l_{\ast}-r_{\ast})(x)v)=&\beta(l(x)v)-\beta(r(x)v)\\
=&l_{\ast}(\alpha(x))\beta(v)-r_{\ast}(\alpha(x))\beta(v)=(l_{\ast}-r_{\ast})(\alpha(x))\beta(v),
\end{align*}
and also
\begin{align*}
&(l_{\ast}-r_{\ast})(\alpha(x))\circ(l_{\ast}-r_{\ast})(y)v-(l_{\ast}-r_{\ast})(\alpha(y))\circ(l_{\ast}-r_{\ast})(x)v\\
&\quad =l_{\ast}(\alpha(x))l_{\ast}(y)v-r_{\ast}(\alpha(x))l_{\ast}(y)v-l_{\ast}(\alpha(x))r_{\ast}(y)v+r_{\ast}(\alpha(x))r_{\ast}(y)v\\
&\quad \quad -l_{\ast}(\alpha(y))l_{\ast}(x)v+l_{\ast}(\alpha(y))r_{\ast}(x)v+r_{\ast}(\alpha(y))l_{\ast}(x)v-r_{\ast}(\alpha(y))r_{\ast}(x)v\\
&\quad =\Big(l_{\ast}(\alpha(x))l_{\ast}(y)v-l_{\ast}(\alpha(y))l_{\ast}(x)v\Big)\\
&\quad\quad +\Big(r_{\ast}(\alpha(y))l_{\ast}(x)v-l_{\ast}(\alpha(x))r_{\ast}(y)v-r_{\ast}(\alpha(y))r_{\ast}(x)v\Big)\\
&\quad\quad -\Big(r_{\ast}(\alpha(x))l_{\ast}(y)v-r_{\ast}(\alpha(x))r_{\ast}(y)v-l_{\ast}(\alpha(y))r_{\ast}(x)v\Big)\\
&\quad =l_{\ast}(\{x,y\})\beta(v)-r_{\ast}(x\ast y)\beta(v)+r_{\ast}(y\ast x)\beta(v)\\
&\quad =l_{\ast}(\{x,y\})\beta(v)-r_{\ast}(\{x,y\})\beta(v)=\rho(\{x,y\})\beta(v).
\end{align*}
Therefore, \eqref{repLie1} and \eqref{repLie3} are satisfied.
\end{proof}
\begin{thm}[\cite{SunLi2017:parakahlerhomhomleftsymetric}]
Let $(A,\ast_A,\alpha)$ and $(B,\ast_{B},\beta)$ be two  Hom-pre-Lie algebras. Suppose that there are linear maps $l_{\ast_A},r_{\ast_A}:A\rightarrow gl(B)$
and $l_{\ast_B},r_{\ast_B}:B\rightarrow gl(A)$ such that $(l_{\ast_A},r_{\ast_A},\beta,B)$ is a bimodule of $A$ and $(l_{\ast_B},r_{\ast_B},\alpha,A)$ is a bimodule of $B$ satisfying for any $x,y\in A,~a,b\in B$ and for
\begin{align*}
\{x,y\}_A=x\ast y-y\ast x, \quad \rho_A=l_{\ast_A}-r_{\ast_A},\\
\{a,b\}_B=a\ast b-b\ast a, \quad \rho_B=l_{\ast_B}-r_{\ast_B},
\end{align*}
the following properties:
\begin{align}
\begin{split}
r_{\ast_A}(\alpha(x))\{a,b\}_B&=r_{\ast_A}(r_{\ast_B}(b)x)\beta(a)-r_{\ast_A}(l_{\ast_B}(a)x)\beta(b)\\
&+\beta(a)\ast_Br_{\ast_A}(x)b)-\beta(b)\ast_B (r_{\ast_A}(x)a),
\end{split}
\label{Lie11}\\
\begin{split}
l_{\ast_A}(\alpha(x))(a\ast_B b)&=(\rho_A(x)a)\ast_B\beta(b)-l_{\ast_A}(\rho_B(a)x)\beta(b)\\
&+\beta(a)\ast_B(l_{\ast_A}(x)b)+r_{\ast_A}(r_{\ast_B}(b)x)\beta(a),
\end{split}
\label{Lie12}\\
\begin{split}
r_{\ast_B}(\beta(a))\{x,y\}_B&=r_{\ast_B}(r_{\ast_A}(y)xa)\alpha(x)-r_{\ast_B}(l_{\ast_A}(x)a)\alpha(y)\\
&+\alpha(x)\ast_A r_{\ast_B}(a)y)-\alpha(y)\ast_A (r_{\ast_B}(a)x),
\end{split}
\label{Lie13}\\
\begin{split}
l_{\ast_B}(\beta(a))(x\ast_A y)&=(\rho_B(a)x)\ast_A\alpha(y)-l_{\ast_B}(\rho_A(x)a)\alpha(y)\\
&+\alpha(x)\ast_A(l_{\ast_B}(a)y)+r_{\ast_B}(r_{\ast_A}(y)a)\alpha(x).
\end{split}
\label{Lie14}
\end{align}
Then, $(A,B,l_{\ast_A},r_{\ast_A},\beta,l_{\ast_B},r_{\ast_B},\alpha)$ is called a matched pair of
Hom-pre-Lie algebras. In this case, there exists a Hom-pre-Lie algebra structure on the vector
space $A\oplus B$ of the underlying vector spaces of $A$ and $B$ given by
\begin{align*}
(x + a) \ast (y + b)&=x \ast_A y + (l_{\ast_A}(x)b + r_{\ast_A}(y)a)+a \ast_B b + (l_{\ast_B}(a)y + r_{\ast_B}(b)x), \\
(\alpha\oplus\beta)(x + a)&=\alpha(x) + \beta(a).
\end{align*}
\end{thm}

We denote this Hom-pre-Lie algebra by $A\bowtie^{l_{\ast_A},r_{\ast_A},\beta}_{l_{\ast_B},r_{\ast_B},\alpha}B$.
\begin{prop}\label{Matchedd1}
Let $(A, B, l_{\ast_{A}}, r_{\ast_{A}}, \beta,
l_{\ast_{B}}, r_{\ast_{B}}, \alpha) $ be a matched pair of Hom-pre-Lie algebras $(A,\ast_A, \alpha)$ and $(B, \ast_B,\beta)$.
Then, $(A, B, l_{\ast_{A}}-r_{\ast_{A}},
\beta ,l_{\ast_{B}}-r_{\ast_{B}},\alpha)$ is a matched pair of the associated
Hom-Lie algebras $(A,\{\cdot,\cdot\}_A, \alpha)$ and $(B,\{\cdot,\cdot\}_B,\beta).$
\end{prop}
\begin{proof}
Let $(A, B, l_{\ast_{A}}, r_{\ast_{A}}, \beta,
l_{\ast_{B}}, r_{\ast_{B}}, \alpha)$ be a matched pair of Hom-pre-Lie algebras $(A,\ast_A, \alpha)$ and $(B, \ast_B,\beta)$. By Proposition \ref{propa1}, the linear maps $l_{\ast_{A}}-r_{\ast_{A}}:A\rightarrow gl(B)$ and  $l_{\ast_{B}}-r_{\ast_{B}}:B\rightarrow gl(A)$ are representations of the underlying Hom-Lie algebras $(A,\{\cdot,\cdot\}_A, \alpha)$ and $(B,\{\cdot,\cdot\}_B,\beta)$, respectively. Therefore, \eqref{Lie1} is equivalent to \eqref{Lie11}-\eqref{Lie12}, and \eqref{Lie2} is equivalent to \eqref{Lie13}-\eqref{Lie14}.
\end{proof}

Next, we define Hom-pre-Lie Poisson algebras and study their properties.
\begin{defn}\label{def pre-poiss}
A Hom-pre-Lie Poisson algebra is a quadruple
$(A,\cdot,\ast,\alpha)$ such that $(A,\cdot,\alpha)$ is a
commutative Hom-associative algebra and $(A,\ast,\alpha)$ is a Hom-pre-Lie algebra
satisfying for all $x,y,z\in A$, the following compatibility conditions:
\begin{eqnarray}
    \label{eq:pre-Poisson 1} (x\cdot y)\ast \alpha(z)&=&\alpha(x)\cdot(y\ast z),\\
    \label{eq:pre-Poisson 2} (x\ast y)\cdot\alpha(z)-(y\ast x)\cdot\alpha(z)&=&\alpha(x)\ast(y\cdot z)-\alpha(y)\ast(x\cdot z).
  \end{eqnarray}
\end{defn}
\begin{ex}
	Consider a two-dimensional  $\mathbb{K}$-linear space $A$ with basis $\{ e_1, e_2\}$.
		Then, $(A,\cdot,\ast,\alpha)$ is a Hom-pre-Lie Poisson algebra, where $$e_1\cdot e_1= ae_2,~
	 e_1\ast e_1=e_1,~e_1\ast e_2=e_2,~\alpha(e_1)=2e_1,~\alpha(e_2)=2e_2,~~a\in\mathbb{K}.$$
		\end{ex}
\begin{defn}
Morphism $f:(A, \cdot, \ast , \alpha)\rightarrow (A', \cdot', \ast ', \alpha')$ of Hom-pre-Lie Poisson algebras are linear maps
$f:A\rightarrow A'$ satisfying for all $x, y\in A$,
\begin{align*}
f(x\cdot y)=f(x)\cdot ' f(y),\quad
f(x\ast y)=f(x)\ast ' f(y),\quad
f\circ \alpha =\alpha '\circ f,
\end{align*}
which can be illustrated by the following standard commutative diagrams:
  $$
\xymatrix{
A\otimes A \ar[d]_{f\times f }\ar[rr]^{\cdot}
               && A  \ar[d]^{f}  \\
A'\otimes A' \ar[rr]^{\cdot'}
               && A' }\quad \xymatrix{
A\otimes A \ar[d]_{f\times f }\ar[rr]^{\ast}
               && A  \ar[d]^{f}  \\
A'\otimes A' \ar[rr]^{\ast '}
               && A' }\quad \xymatrix{
A \ar[d]_{ f }\ar[rr]^{\alpha}
              && A  \ar[d]^{f}  \\
A' \ar[rr]^{\alpha'}
             && A'}
$$
\end{defn}
\begin{thm} \label{Yaupre-poiss}
Let $\mathcal{A}=(A, \cdot ,\ast)$ be a pre-Lie Poisson algebra and
$\alpha :\mathcal{A}\rightarrow \mathcal{A}$ be a pre-Lie Poisson algebras
morphism. Define $\cdot_{\alpha}: A
\times A\rightarrow A$ and $\ast_\alpha: A
\times A\rightarrow A$ for all $x, y\in A$ by
$
x\cdot _{\alpha}y=\alpha (x\cdot y)$ and $x\ast_{\alpha}y=\alpha(x\ast y).
$
Then, $\mathcal{A}_\alpha=(A_\alpha=A, \cdot _{\alpha},\ast_\alpha, \alpha)$ is a Hom-pre-Lie Poisson algebra, called $\alpha$-twist or Yau twist of $\mathcal{A}$. Moreover, assume that $\mathcal{A}'=(A', \cdot',\ast')$ is another pre-Lie Poisson algebra, and
$\alpha':\mathcal{A}'\rightarrow \mathcal{A}'$ is a pre-Lie Poisson algebras morphism.
Let $f:\mathcal{A}\rightarrow \mathcal{A}'$ be a
pre-Lie Poisson algebras morphism satisfying $f\circ \alpha =\alpha
'\circ f$. Then, $f:\mathcal{A}_{\alpha }\rightarrow \mathcal{A}'_{\alpha}$ is a
Hom-pre-Lie Poisson algebras morphism.
\end{thm}
\begin{proof}
We shall only prove \eqref{eq:pre-Poisson 1} in $\mathcal{A}_{\alpha}$, as the other relation is proved analogously. For any $x, y, z \in A$,
\begin{align*}
(x\cdot_{\alpha}y)\ast_{\alpha}\alpha(z)=&\alpha(\alpha(x)\cdot\alpha(y))\ast\alpha^{2}(z)\\
=&(\alpha^{2}(x)\cdot\alpha^{2}(y))\ast\alpha^{2}(z)\\
=&\alpha^{3}(x)\cdot(\alpha^{2}(y)\ast\alpha^{2}(z))
\quad \mbox{(by \eqref{eq:pre-Poisson 1}~in~$\mathcal{A}$)}\\
=&\alpha(x)\cdot_{\alpha}(y\ast_{\alpha} z).
\end{align*}
The second assertion follows from
\begin{align*}
    f(x\cdot_{\alpha} y)=&f(\alpha(x)\cdot\alpha(y))
    =f(\alpha(x))\cdot' f(\alpha(y))
    =\alpha' f(x)\cdot'\alpha' f(y)
    =f(x)\cdot'_{\alpha'} f(y).
\end{align*}
Similarly, $  f(x\ast_{\alpha} y)=f(x)\ast'_{\alpha'} f(y)$.
\end{proof}
\begin{prop}\label{isma}
Let $\mathcal{A}=(A, \cdot ,\ast, \alpha)$ be a
Hom-pre-Lie Poisson algebra and $\alpha': \mathcal{A}\rightarrow \mathcal{A}$ be a morphism.
Then,  $(A, \alpha'\circ\cdot,
\alpha'\circ\ast,\alpha \circ \alpha')$ is a Hom-pre-Lie Poisson
algebra.
\end{prop}
\begin{cor}
If $\mathcal{A}=(A, \cdot,\ast,\alpha)$ is a multiplicative Hom-pre-Lie Poisson
algebra, then for any $n\in\mathbb{N}^{\ast}$,
\begin{enumerate}
\item
The $n{\rm th}$ derived Hom-pre-Lie Poisson algebra of type $1$ of $\mathcal{A}$ is
defined by
$$\mathcal{A}_{1}^{n}=(A,\cdot^{(n)}=\alpha^{n}\circ\cdot,\ast^{(n)}=\alpha^{n}\circ\ast,\alpha^{n+1}).$$
\item
The $n{\rm th}$ derived Hom-pre-Lie Poisson algebra of type $2$ of $A$ is
defined by
$$\mathcal{A}_{2}^{n}=(A,\cdot^{(2^n-1)}=\alpha^{2^n-1}\circ\cdot,\ast^{(2^n-1)}=\alpha^{2^n-1}\circ\ast,\alpha^{2^n}).$$
\end{enumerate}
\end{cor}
\begin{proof}
Apply Theorem \ref{isma} with $\alpha'=\alpha^{n}$ and
$\alpha'=\alpha^{2^n-1}$ respectively.
\end{proof}
\begin{prop}
Let  $(A_1,\cdot_1,\ast_1,\alpha_1)$ and $(A_2,\cdot_2,\ast_2,\alpha_2)$ be two Hom-pre-Lie
Poisson algebras. Define two operations $\cdot$ and $\ast$ on $A_1\otimes A_2$ for all $x_1,y_1\in A_1, x_2,y_2\in A_2$ by
\begin{align}
&(x_1\otimes x_2)\cdot (y_1\otimes y_2)=x_1\cdot_1 y_1\otimes x_2\cdot_2
y_2,\label{eq:tensor1:HompreLie}\\
&(x_1\otimes x_2)\ast(y_1\otimes y_2)=(x_1\ast_1 y_1)\otimes (x_2\cdot_2
y_2)+(x_1\cdot_1 y_1)\otimes (x_2\ast_2 y_2),\label{eq:tensor2:HompreLie}\\
&(\alpha_1\otimes \alpha_2)(x_1\otimes x_2)=\alpha_1(x_1)\otimes\alpha_2(x_2). \label{eq:tensor3:HompreLie}\end{align}
Then, $(A_1\otimes A_2,\cdot, \ast,\alpha_1\otimes\alpha_2)$ is a Hom-pre-Lie Poisson algebra.
\end{prop}
\begin{thm}
Let $\mathcal{A}=(A,\cdot,\ast,\alpha)$ be a Hom-pre-Lie Poisson
algebra. For any $x,y \in A$, let $\{x, y\}=x\ast y-y\ast x$.
Then, $\mathcal{A}^{c}=(A,\cdot,\{\cdot,\cdot\},\alpha)$ is a
transposed Hom-Poisson algebra, called the sub-adjacent transposed Hom-Poisson algebra of $(A,\cdot,\ast,\alpha)$.
\end{thm}
\begin{proof}
By Lemma \ref{lem1}, $(A,\{\cdot,\cdot\},\alpha)$ is a Hom-Lie algebra. The
transposed Hom-Leibniz identity is proved as follows:
\begin{align*}
&2\alpha(z)\cdot(x\ast y-y\ast x)-((z\cdot x)\ast\alpha(y))\\
&\quad -\alpha(y)\ast(z\cdot x))-(\alpha(x)\ast(z\cdot y)-(z\cdot y)\ast \alpha(x))\\
&=2\alpha(z)\cdot(x\ast y)-2\alpha(z)\cdot(y\ast x)-(z\cdot x)\ast \alpha(y)\\
&\quad +\alpha(y)\ast(z\cdot x)-\alpha(x)\ast(z\cdot y)+(z\cdot y)\ast\alpha(x)\\
&=(\alpha(z)\cdot(x\ast y)-(z\cdot x)\ast \alpha(y))+((z\cdot y)\ast\alpha(x)\\
&\quad -\alpha(z)\cdot(y\ast x))+(\alpha(z)\cdot(x\ast y)-\alpha(z)\cdot(y\ast x)\\
&\quad -\alpha(y)\ast(z\cdot x)+\alpha(x)\ast(z\cdot y))\\
&=0+0+0=0.
\quad \mbox{(by \eqref{eq:pre-Poisson 1}, \eqref{eq:pre-Poisson 2} and by commuatativity)}
\end{align*}
Thus, $(A,\cdot,\{\cdot,\cdot\},\alpha)$ is a
Hom-transposd Poisson algebra.
\end{proof}
The relation existing between a transposed Hom-Poisson algebra and Hom-pre-Lie Poisson algebra, as illustrated by the following diagram:
$$
\xymatrix{ \mbox{Hom-comm ass alg+ Hom-pre-Lie alg} \ar[rr] \ar[dd]^{[x,y]=x\ast y-y\ast x}&& \mbox{Hom-pre-Lie Poisson alg} \ar[dd]_{[x,y]=}^{x\ast y-y\ast x}\\
&& \\
\mbox{Hom-comm ass alg+Hom-Lie alg} \ar[rr] && \mbox{transposed Hom-Poisson alg.}
}
$$

In the following we introduce the notions of bimodule and matched pair of Hom-pre-Lie Poisson algebras and related relevant properties are also given
\begin{defn}\label{def bim}
Let $\mathcal{A}=(A, \cdot,\ast, \alpha)$ be a Hom-pre-Lie Poisson algebra. A bimodule of $\mathcal{A}$ is a $5$-tuple $(s,l_{\ast}, r_{\ast},\beta,V)$ such that
$( l_{\ast}, r_{\ast},
\beta,V)$ is a bimodule of the Hom-pre-Lie algebra $(A, \ast, \alpha)$ and $( s,
\beta,V)$ is a bimodule of the commutative Hom-associative algebra $(A, \cdot, \alpha)$
satisfying for all $ x, y \in A$, $v\in V$ and $\{x,y\} = x
\ast y - y \ast x,~ \rho = l_{\ast} - r_{\ast}$,
\begin{eqnarray}
l_\ast(x\cdot y)\beta(v)&=&s(\alpha(x))l_\ast(y)v,\label{3,4}\\
r_\ast(\alpha(y))s(x)v&=&s(x\ast y)\beta(v),\label{3,5}\\
r_\ast(\alpha(y))s(x)v&=&s(\alpha(x))r_\ast(y)v,\label{3,6}\\
s(\{x,y\})\beta(v)&=&l_\ast(\alpha(x))s(y)v-l_\ast(\alpha(y))s(x)v,\label{3,7}\\
s(\alpha(y))\rho(x)v&=&l_\ast(\alpha(x))s(y)v-r_\ast(x\cdot y)\beta(v).\label{3,8}
\end{eqnarray}
\end{defn}
\begin{prop}
If $(s,l_{\ast}, r_{\ast},\beta, V)$ is a bimodule of a Hom-pre-Lie Poisson algebra $(A,\cdot,\ast, \alpha),$
then, there is a Hom-pre-Lie Poisson algebra structure on $(A\oplus V,\cdot',\ast',\alpha+\beta),$ where $(A\oplus V,\cdot',\alpha+\beta)$ is the semi-direct product commutative Hom-associative algebra $A\ltimes_{s,\alpha,\beta} V$, and $(A\oplus V,\ast',\alpha+\beta)$ is the semi-direct product Hom-Lie algebra $A\ltimes_{l_{\ast},r_{\ast},\alpha,\beta} V$.
\end{prop}
\begin{proof}
Let $(s,l_{\ast}, r_{\ast}, \beta, V)$ be a bimodule of a Hom-pre-Lie Poisson algebra
$(A,\cdot,\ast, \alpha).$
By Proposition \ref{ass1} and Proposition \ref{bimodpre-Lie}, $(A\oplus V,\cdot',\alpha+\beta)$ is a commutative Hom-associative algebra and $(A\oplus V,\ast',\alpha+\beta)$ is a Hom-pre-Lie algebra respectively.
Now, we prove the axioms \eqref{eq:pre-Poisson 1} and \eqref{eq:pre-Poisson 2} in $A\oplus V$. For any
$x_{1},x_{2},x_{3}\in A$ and $v_1, v_2, v_3\in V$,
\begin{align*}
& ((x_1+v_1)\cdot'(x_2+v_2))\ast'(\alpha+\beta)(x_3+v_3)\\
&\quad =(x_1\cdot x_2+s(x_1)v_2+s(x_2)v_1)\ast'(\alpha(x_3)+\beta(v_3))\\
&\quad =(x_1\cdot x_2)\ast\alpha(x_3)+l_{\ast}(x_1\cdot x_2)\beta(v_3)
+r_{\ast}(\alpha(x_3))s(x_1)v_2+r_{\ast}(\alpha(x_3))s(x_2)v_1,
\\
& (\alpha+\beta)(x_1+v_1)\cdot'((x_2+v_2)\ast'(x_3+v_3))\\
&\quad =(\alpha(x_1)+\beta(v_1))\cdot'(x_2\ast x_3+l_{\ast}(x_2)v_3+r_{\ast}(x_3)v)\\
&\quad =\alpha(x_1)\cdot(x_2\ast x_3)+s(\alpha(x_1))l_{\ast}(x_2)v_3+s(\alpha(x_1))r_{\ast}(x_3)v+s(x_2\ast x_3)\beta(v_1).
\end{align*}
By \eqref{3,4}-\eqref{3,6},
and \eqref{eq:pre-Poisson 1} in $A$,
\begin{align*}
&((x_1+v_1)\cdot'(x_2+v_2))\ast'(\alpha+\beta)(x_3+v_3) \\
&\quad \quad \quad \quad \quad \quad = (\alpha+\beta)(x_1+v_1)\cdot'((x_2+v_2)\ast'(x_3+v_3)),\\
& ((x_1+v_1)\ast'(x_2+v_2))\cdot'(\alpha+\beta)(x_3+v_3)-((x_2+v_2)\ast'(x_1+v_1))\cdot'(\alpha+\beta)(x_3+v_3)\\
&\quad =(x_1\ast x_2+l_{\ast}(x_1)v_2+r_\ast(x_2)v_1)\cdot'(\alpha(x_1)+\beta(v_2))\\&\quad\quad-(x_2\ast x_1+l_\ast(x_2)v_1+r_\ast(x_1)v_2)\cdot'(\alpha(x_3)+\beta(v_3))\\
&\quad =(x_1\ast x_2)\cdot\alpha(x_3)+s(x_1\ast x_2)\beta(v_3)+s(\alpha(x_3))l_\ast(x_1)v_2+s(\alpha(x_3))r_\ast(x_2)v_1\\
&\quad\quad-(x_2\ast x_1)\cdot\alpha(x_3)-s(x_2\ast x_1)\beta(v_3)-s(\alpha(x_3))l_\ast(x_2)v_1-s(\alpha(x_3))r_\ast(x_1)v_2,
\\
&(\alpha+\beta)(x_1+v_1)\ast'((x_2+v_2)\cdot'(x_3+v_3))-(\alpha+\beta)(x_2+v_2)\ast'((x_1+v_1)\cdot'(x_3+v_3))\\
&\quad=(\alpha(x_1)+\beta(v_1))\ast'(x_2\ast x_3+s(x_2)v_3+s(x_3)v_2)\\
&\quad\quad-(\alpha(x_2)+\beta(v_2))\ast'(x_1\cdot x_3+s((x_1)v_3+s(x_3)v_1)\\
&\quad=\alpha(x_1)\ast(x_2\cdot x_3)+s(\alpha(x_1))s(x_2)v_3+l_\ast(\alpha(x_1))s(x_3)v_2+r_\ast(x_1\cdot x_3)\beta(v_2)\\
&\quad\quad-\alpha(x_2)\ast(x_1\cdot x_3)-l_\ast(\alpha(x_2))s(x_1)v_3-l_\ast(\alpha(x_2))s(x_3)v_1-r_\ast(x_1\cdot x_3)\beta(v_2).
\end{align*}
By \eqref{3,7}-\eqref{3,8}, and \eqref{eq:pre-Poisson 2} in $A$,
\begin{align*}&((x_1+v_1)\ast'(x_2+v_2))\cdot'(\alpha+\beta)(x_3+v_3) -((x_2+v_2)\ast'(x_1+v_1))\cdot'(\alpha+\beta)(x_3+v_3)
\\&=(\alpha+\beta)(x_1+v_1)\ast'((x_2+v_2)\cdot'(x_3+v_3))\\
& \quad \quad \quad \quad \quad \quad -(\alpha+\beta)(x_2+v_2)\ast'((x_1+v_1)\cdot'(x_3+v_3)).
\qedhere \end{align*}
\end{proof}
We denote such Hom-pre-Lie Poisson algebra by $A\times_{s, l_{\ast},r_{\ast},\alpha, \beta} V$.

\begin{ex}
Let $(A,\cdot,\ast,\alpha)$ be a Hom-pre-Lie Poisson algebra.
Then, $(S,L_{\ast},R_{\ast},\alpha,A)$ is called a regular
bimodule of  $(A,\cdot,\ast,\alpha)$, where for all $(x,y)\in A\times A$,
$$S(x)y=x\cdot y, \quad L_{\ast}(x)y=x\ast y, \quad R_{\ast}(x)y=y\ast x.$$
\end{ex}
\begin{prop}
If $f:\mathcal{A}=(A,\cdot_1,\ast_1,\alpha)\longrightarrow(A',\cdot_2,\ast_2,\beta)$ is a Hom-pre-Lie Poisson algebras morphism, then
$(s_1,l_{\ast_1},r_{\ast_1},\beta,A')$
becomes a bimodule of $\mathcal{A}$ via $f$, that is, for all $(x,y)\in A\times A'$,
$s_1(x)y=f(x)\cdot_2 y,$ \ $l_{\ast_1}(x)y=f(x)\ast_2 y,$ \ $r_{\ast_1}(x)y=y \ast_2 f(x).$
\end{prop}
\begin{proof}
We prove the axioms \eqref{3,4}-\eqref{3,8}, as the others being proved similarly. For any
$x,y\in A$ and $z\in A'$,
\begin{align*}
    l_{\ast_1}(x\cdot_1 y)\beta(z)&=f(x\cdot_1 y)\ast_2\beta(z)\\&=(f(x)\cdot_2f(y))\ast_2\beta(z)\\
    &=\beta f(x)\cdot_2(f(y)\ast_2 z)\\&=f(\alpha(x))\cdot_2(f(y)\ast_2 z)\\
    &=s_1(\alpha(x))(f(y)\ast_2 z)\\&=s_1(\alpha(x))l_{\ast_1}(y)z,\\
    r_{\ast_1}(\alpha(y))s_1(x)(z)&=r_{\ast_1}(\alpha(y))(z\cdot_2 f(x))\\&=(z\cdot_2 f(x))\ast_2 f(\alpha(y))\\
    &=(z\cdot_2 f(x))\ast_2\beta(f(y))\\
    &=(f(x)\ast_2 f(y))\cdot_2 \beta(z)\\
   &= f(x\ast_1 y)\cdot_2\beta(z)\\
    &=s_1(x\ast_1 y)\beta(z),\\
    r_{\ast_1}(\alpha(y))s_1(x)z&=s_1(x)z\ast_2 f(\alpha(y))\\
    &=s_1(x)z\ast_2 \beta(f(y))\\
    &=(f(x)\cdot_2 z)\ast_2 \beta(f(y))\\
    &=\beta(f(x))\cdot_2(z\ast_2 f(y))\\
    &=f(\alpha(x))\cdot_2(z\ast_2 f(y))\\
    &=s_1(\alpha(x))(z\ast_{2}f(y))\\
    &=s_1(\alpha(x))r_{\ast_1}(y)z,\\
    s_{\ast_1}(\{x,y\}_1)\beta(z)&=s_1(x\ast_1 y-y\ast_1 x)\beta(z)\\
    &=s_1(x\ast_1 y)\beta(z)-s_1(y\ast_1 x)\beta(z)\\
    &=f(x\ast_1 y)\cdot_2\beta(z)-f(y\ast_1 x)\cdot_2\beta(z)\\
    &=(f(x)\ast_2 f(y))\cdot_2 \beta(z)-(f(y)\ast_2 f(x))\cdot_2\beta(z)\\
    &=\beta(f(x))\ast_2(f(y)\cdot_2 z)-\beta(f(y))\ast_2(f(x)\cdot_2 z)\\
    &=f(\alpha(x))\ast_2(f(y)\cdot_2 z)-f(\alpha(y))\ast_2(f(x)\cdot_2 z)\\
    &=l_{\ast_1}(\alpha(x))(f(y)\cdot_2 z)-l_{\ast_1}(\alpha(y))(f(x)\cdot_2 z)\\
    &=l_{\ast_1}(\alpha(x))s_1(y) z-l_{\ast_1}(\alpha(y))s_1(x)z,\\
    s_1(\alpha(y))\rho_1(x)v&=s_1(\alpha(y))(l_{\ast_1}-r_{\ast_1})(x)z\\
    &=s_1(\alpha(y))l_{\ast_1}(x)z-s_1(\alpha(y))r_{\ast_1}(x)z\\
    &=f(\alpha(y))\cdot_2 l_{\ast_1}(x)z-f(\alpha(y))\cdot_2 r_{\ast_1}(x)z\\
    &=\beta(f(y))\cdot_2(f(x)\ast_2 z)-\beta f(y)\cdot_2(z\ast_2 f(x))\\
    &=\beta(f(x))\cdot_2(z\cdot_2 f(y))-\beta(z)\ast_2(f(x)\cdot_2 f(y))\\
    &=f(\alpha(x))\ast_2(z\cdot_2 f(y))-\beta(z)\ast_2(f(x)\cdot_1 y)\\
    &=l_{\ast_1}(\alpha(x))s_1(y)z-r_{\ast_1}(x\cdot_1 y)\beta(z).
\qedhere \end{align*}
\end{proof}
\begin{cor}\label{corooo}
Let $(s,l_{\ast},r_{\ast}, \beta, V)$ be a bimodule of a Hom-pre-Lie Poisson algebra
$(A, \cdot,\ast, \alpha)$. Let $(A, \cdot,\{\cdot,\cdot\}, \alpha)$ be the
subadjacent of $(A, \cdot,\ast, \alpha)$.
Then, $( s,l_\ast-r_\ast,\beta, V)$ is a
representation of $(A, \cdot,\{\cdot,\cdot\}, \alpha)$.
\end{cor}
\begin{proof}
It follows from the relation between the Hom-pre-Lie Poisson algebra
and the associated transposed Hom-Poisson algebra. More precisely
by Poposition \ref{propa1}, we deduce that $(l_\ast-r_\ast,\beta,V)$ is a
representation of $(A, \{\cdot,\cdot\}, \alpha)$. Now, the rest, it is easy (in a similar way as for Poposition \ref{propa1}) to verify the axioms \eqref{ismail.1}-\eqref{ismail.2}
\end{proof}
\begin{thm}\label{mamm}
Let $\mathcal{A}=(A,\cdot,\ast,\alpha)$ be a
Hom-pre-Lie Poisson algebra, and let
$V_{\beta}=(s,l_{\ast},r_{\ast},\beta,V)$ be a bimodule of $\mathcal{A}$. Let $\alpha'$
be a morphism of $\mathcal{A}$ such that the maps $\alpha$ and $\alpha'$ commute,
and let $\beta'$ be a linear map of $V$ such that the maps $\beta$ and $\beta'$
commute. Furthermore, suppose that
$$\left\{
   \begin{array}{lllllll}
    \beta'\circ s=(s\circ\alpha')\beta',\\
      \beta'\circ l_\ast=(l_\ast\circ\alpha')\beta',\\
       \beta'\circ r_\ast=(r_\ast\circ\alpha')\beta',
   \end{array}
 \right.$$
$\mathcal{A}_{\alpha'}$ is the Hom-pre-Lie Poisson algebra
$(A,\cdot_{\alpha'}, \ast_{\alpha'},\alpha\alpha')$,
and
$V_{\beta'}=(\widetilde{s},\widetilde{l}_{\ast},\widetilde{r}_{\ast},\beta\beta',V)$,
where
$
\widetilde{s}=(s\circ\alpha')\beta',~\widetilde{l}_{\ast}
=(l_{\ast}\circ\alpha')\beta',~\widetilde{r}_{\ast}=(r_{\ast}\circ\alpha')\beta'.
$
Then, $V_{\beta'}$ is a bimodule of $\mathcal{A}_{\alpha'}$.
\end{thm}
\begin{proof}
We prove only the axioms \eqref{3,4}-\eqref{3,8}, as the  others being proved similarly. For any
$x,y\in A$ and $v\in V$,
\begin{align*}
\widetilde{l}_\ast(x\ast_\alpha' y)\beta\beta'(v)&=l_\ast(\alpha'^{2}(x)\ast\alpha'^{2}(y))\beta\beta'^{2}(v)\\
&=s(\alpha\alpha'^{2}(x))l_{\ast}(\alpha'^{2}(y))\beta'^{2}(v)
\quad \mbox{(by \eqref{3,4} in $\mathcal{A}$)}\\
&=\widetilde{s}(\alpha\alpha'(x))\widetilde{l}_{\ast}(y)v,\\
\widetilde{r_\ast}(\alpha\alpha'(y))\widetilde{s}(x)v&=\widetilde{r_\ast}(\alpha\alpha'(y))s(\alpha'(x))\beta'(v)\\
&=r_\ast(\alpha\alpha'^{2}(y))\beta's((\alpha'(x))\beta'(v)\\
&=r_\ast(\alpha\alpha'^{2}(y))s((\alpha'^{2}(x))\beta'^{2}(v)\\
&=s(\alpha'^{2}(x)\ast\alpha'^{2}(y))\beta\beta'^{2}(v)\\
&=s(\alpha'(x)\ast_{\alpha'}\alpha'(y))\beta\beta'^{2}(v)\\
&=s(\alpha'(x)\ast_{\alpha'}\alpha'(y))\beta\beta'^{2}(v)\\
&=\widetilde{s}(x\ast_{\alpha'}y)\beta\beta'(v),\\
\widetilde{r_\ast}(\alpha\alpha'(y))\widetilde{s}(x)v&=\widetilde{r_\ast}(\alpha\alpha'(y))s(\alpha'(x))\beta'(v)\\
&=r_\ast(\alpha\alpha'^{2}(y))s(\alpha'^{2}(x))\beta'^{2}(v)\\
&=s(\alpha\alpha'^{2}(x))r_\ast(\alpha'^{2}(y))\beta'^{2}(v)\\
&=s(\alpha\alpha'^{2}(x))\widetilde{r_\ast}(\alpha'(y))\beta'(v)\\
&=\widetilde{s}(\alpha\alpha'(x))\widetilde{r_\ast}(y)v,\\
\widetilde{s}(\{x,y\}_{\alpha'})\beta\beta'(v)&=\widetilde{s}(\{\alpha'(x),\alpha'(y)\})\beta\beta'(v)\\
&=s(\{\alpha'^{2}(x),\alpha'^{2}(y)\})\beta\beta'^{2}(v)\\
&=l_{\ast}(\alpha\alpha'^{2}(x))s(\alpha'^{2}(y))\beta'^{2}(v)\\
&-l_{\ast}(\alpha\alpha'^{2}(y))s(\alpha'^{2}(x))\beta'^{2}(v)\\
&=l_{\ast}(\alpha\alpha'^{2}(x))\widetilde{s}(\alpha'(y))\beta'(v)\\
&-l_{\ast}(\alpha\alpha'^{2}(y))\widetilde{s}(\alpha'(x))\beta'(v)\\
&=\widetilde{l_\ast}(\alpha\alpha'(x))\widetilde{s}(y)v\\
&-\widetilde{l_\ast}(\alpha\alpha'(y))\widetilde{s}(x)v,\\
\widetilde{s}(\alpha\alpha'(y))\widetilde{\rho}(x)v&=\widetilde{s}(\alpha\alpha'(y))\rho(\alpha'(x))\beta'(v)\\
&=\widetilde{s}(\alpha\alpha'^{2}(y))\rho(\alpha'^{2}(x))\beta'^{2}(v)\\
&=l_\ast(\alpha\alpha'^{2}(y))\rho(\alpha'^{2}(x))s(\alpha'^{2}(y))\beta'^{2}(v)\\
&-r_\ast(\alpha'^{2}(x)\cdot\alpha'^{2}(y))\beta\beta'^{2}(v)\\
&=l_\ast(\alpha\alpha'^{2}(y))\rho(\alpha'^{2}(x))\widetilde{s}(\alpha'(y))\beta'(v)\\
&-\widetilde{r_\ast}(\alpha'(x)\cdot\alpha'(y))\beta\beta'(v)\\
&=\widetilde{l_\ast}(\alpha\alpha'(x))\widetilde{s}(y)v\\
&-\widetilde{r_\ast}(x\cdot_\alpha' y)\beta\beta'(v).
\qedhere \end{align*}
\end{proof}
Taking $\alpha'=\alpha^p$ and $\beta'=\beta^q$ leads to the following statement.
\begin{cor}
Let $\mathcal{A}=(A,\cdot,\ast,\alpha)$ be a multiplicative Hom-pre-Lie Poisson algebra, and
$(s,l_{\ast},r_{\ast},\beta,V)$ a bimodule of
$\mathcal{A}$. Then, $V_{\beta^{q}}$ is a bimodule of $A_{\alpha^{p}}$ for any
nonnegative integers $p$ and $q$.
\end{cor}
\begin{thm}
Let $(A, \cdot_{A},\ast_A, \alpha)$ and $(B,  \cdot_{B}, \ast_B,\beta)$
 be two Hom--pre-Lie Poisson algebras. Suppose that there are linear maps
$ s_{A},l_{\ast_A},r_{\ast_A} : A \rightarrow gl(B),$
and $ s_{B},l_{\ast_B},r_{\ast_B} : B \rightarrow gl(A)$
such that
$A\bowtie^{l_{\ast_A},r_{\ast_A},\beta}_{l_{\ast_B},r_{\ast_B},\alpha}B$ is a matched pair of Hom-pre-Lie algebras and $A\bowtie^{s_A,\beta}_{s_B,\alpha}B$ is a matched pair of commutative Hom-associative algebra, and for all $x, y \in A,~ a, b \in B$, and for
$$\{x,y\}_A=x\ast_A y-y\ast_A x,~\{a,b\}_B=a\ast_ B-b\ast_B a,~
 \rho_A = l_{\ast_{A}}-r_{\ast_{A}},~\rho_B= l_{\ast_{B}}   -r_{\ast_{B}},$$
the following equalities hold:
\begin{align}\label{eq_matched_pre_1}
r_{\ast_A}(\alpha(x))(a\cdot_B b)&=\beta(a)\cdot_B(r_{\ast_A}(x)b)+s_A(l_{\ast_B}(b)x)\beta(a),
\\
\label{eq_matched_pre_2}
(s_A(x)a)\ast_B\beta(b)+l_{\ast_A}(s_B(a)x)\beta(b)&=s_A(\alpha(x))(a\ast_B b),
\\
\label{eq_matched_pre_3}
l_{\ast_A}(s_B(a)x)\beta(b)+(s_A(x)a)\ast_B\beta(b)&=\beta(a)\cdot_B(l_{\ast_A}(x)b)+s_A(r_B(b)x)\beta(a),  \\                \label{eq_matched_pre_4}
\begin{split}
s_{A}(\alpha(x))(\{a\ast_B b\}_B) &=\beta(a)\ast_B(s_A(x)b)+r_{\ast_A}(s_B(b)x)\beta(a)\\
&-\beta(b)\ast_B(s_A(x)a)-r_A(s_B(a)x)\beta(b),
\end{split}
\\
\label{eq_matched_pre_5}
\begin{split}
(\rho_A(x)a)\cdot_B\beta(b)-s_A(\rho_B(a)x)\beta(b)&=l_{\ast_A}(\alpha(x))(a\cdot_B b) \\
&-\beta(a)\ast_B(l_{\ast_A}(x)b)-r_{\ast_A}(s_B(b)x)\beta(a),
\end{split}
\\
\label{eq_matched_pre_6}
r_{\ast_B}(\beta(a))(ax\cdot_A y)&=\alpha(x)\cdot_A(r_{\ast_B}(a)y)+s_B(l_{\ast_A}(y)a)\alpha(x),
\\
\label{eq_matched_pre_7}
(s_B(a)x)\ast_A\alpha(y)+l_{\ast_B}(s_A(x)a)\alpha(y)&=s_B(\beta(a))(x\ast_A y),
\\
\label{eq_matched_pre_8}
l_{\ast_B}(s_A(x)a)\alpha(y)+(s_B(a)x)\ast_A\alpha(y)&=\alpha(x)\cdot_A(l_{\ast_B}(a)y)+s_B(r_A(y)a)\alpha(x),
\\
\label{eq_matched_pre_9}
\begin{split}
s_{B}(\beta(a))(\{x\ast_A y\}_A)&=\alpha(x)\ast_A(s_B(a)y)+r_{\ast_B}(s_A(y)a)\alpha(x)\\
&- \alpha(y)\ast_A(s_B(a)x)-r_B(s_A(x)a)\alpha(y),
\end{split}
\\
\label{eq_matched_pre_10}
\begin{split}
(\rho_B(a)x)\cdot_A\alpha(x)-s_B(\rho_A(x)a)\alpha(y)&=l_{\ast_B}(\beta(a))(x\cdot_A y)\\
&-\alpha(x)\ast_A(l_{\ast_B}(a)y)-r_{\ast_B}(s_A(y)a)\alpha(x).
\end{split}
\end{align}
Then,
 $(A, B, s_A,l_{\ast_{A}}, r_{\ast_{A}}, \beta, s_B, l_{\ast_{B}}, r_{\ast_{B}},\alpha)$ is called a matched pair of Hom-pre-Lie Poisson algebras.
 In this case, there exists a  Hom-pre-Lie Poisson algebra structure on the direct sum $ A \oplus B $
            \end{thm}
\begin{proof}
The proof is obtained in a similar way as for Theorem \ref{matched ass}.
\end{proof}

Let $ A \bowtie^{s_A,l_{\ast_{A}}, r_{\ast_{A}}, \beta}_{s_B,l_{\ast_{B}},
r_{\ast_{B}}, \alpha} B $ denote this Hom-pre-Lie Poisson algebra.

\begin{cor}
Let $(A, B, s_A,l_{\ast_{A}}, r_{\ast_{A}}, \beta,
 s_B,l_{\ast_{B}}, r_{\ast_{B}}, \alpha) $ be a matched pair of Hom-pre-Lie Poisson algebras $ (A, \cdot_A,\ast_{A}, \alpha) $ and $  (B, \cdot_B,\ast_{B}, \beta) $.

Then, $(A, B,s_A ,l_{\ast_{A}}-r_{\ast_{A}},\beta,
s_B,l_{\ast_{B}}-r_{\ast_{B}},\alpha )$ is a matched pair of the associated
transposed Hom-Poisson algebras $(A,\cdot_{A},\{\cdot,\cdot\}_A, \alpha)$ and $(B, \cdot_{B},\{\cdot,\cdot\}_B,\beta)$.
\end{cor}
\begin{proof}
Let $(A, B, s_{A},l_{\ast_{A}}, r_{\ast_{A}}, \beta,
 s_{B},l_{\ast_{B}}, r_{\ast_{B}}, \alpha) $ be a matched pair of Hom-pre-Lie Poisson algebras $ (A, \cdot_A,\ast_{A}, \alpha) $ and $  (B, \cdot_{B},\ast_{B}, \beta) $. Then, $(A, B, s_{A}, \beta,
 s_{B}, \alpha) $ is a matched pair of commutative Hom-associative algebras $ (A, \cdot_A, \alpha) $ and $  (B, \cdot_{B}, \beta)$, and by Proposition \ref{Matchedd1}, $(A, B,  l_{\ast_{A}}-r_{\ast_{A}},\beta,
l_{\ast_{B}}-r_{\ast_{B}},\alpha )$ is a matched pair of the associated Hom-Lie algebras $(A,\{\cdot,\cdot\}_A, \alpha)$ and $(B,\{\cdot,\cdot\}_B,\beta)$. Besides, by Corollary \ref{corooo}, the linear maps $s_{A},l_{\ast_A},r_{\ast_A}:A\rightarrow gl(B)$ and $s_{B},l_{\ast_B},r_{\ast_B}:B\rightarrow gl(A)$ are representations of the underlying transposed Hom-Poisson algebras  $(B,\cdot_B,\{\cdot,\cdot\}_B,\beta)$ and $(A,\cdot_A,\{\cdot,\cdot\}_A, \alpha)$ respectively. Therefore, \eqref{101}-\eqref{102} are
equivalent to \eqref{eq_matched_pre_1}-\eqref{eq_matched_pre_5}, and \eqref{103}-\eqref{104} are
equivalent to \eqref{eq_matched_pre_6}-\eqref{eq_matched_pre_10}.
\end{proof}
\section{\texorpdfstring{$\mathcal{O}$}{}-operators of transposed Hom-Poisson algebras}\label{sec5}
In this section, we introduce and study the notion of an $\mathcal{O}$-operator of transposed Hom-Poisson algebras generalizing the notion of Rota-Baxter operators. Their properties and relationship with
Hom-pre-Lie Poisson algebras is also described.

\begin{defn}
Let $(A, \cdot, \alpha)$ be a commutative Hom-associative algebra and $(s, \beta, V)$ be a bimodule. Then, a  linear map $ T : V \rightarrow A $
is called an $ \mathcal{O} $-operator associated to $(s, \beta, V)$,  if $ T $ satisfies
for all $u, v \in V,$
\begin{align}
&\alpha T= T\beta, \label{morphismm} \\
&T(u)\cdot T(v) = T(s(T(u))v + s(T(v))u).
\end{align}
\end{defn}
\begin{lem}\label{lem silv double}
Let $(A, \cdot, \alpha)$ be a commutative Hom-associative algebra, and $(l, r, \beta, V) $ be a bimodule.
Let $ T : V \rightarrow A $ be an $ \mathcal{O}$-operator associated to $(s, \beta, V)$. Then, there exists a commutative Hom-associative algebra structure on $ V $ given for all $u, v \in V$ by
\begin{eqnarray*}
 u \diamond v = s(T(u))v + s(T(v))u.
\end{eqnarray*}
\end{lem}
\begin{proof}
First, we prove the commutativity condition.
For all $u, v \in V$ , $$u \diamond v = s(T(u))v+s(T(v))u = s(T(v))u + s(T(u))v = v \diamond u.$$

Now, we prove the Hom-associativity condition. For all $u,v,w \in V$,
\begin{align*}
   & (u\diamond v)\diamond\beta(w)=(s(T(u))v+s(T(v))u)\diamond\beta(w)\\
    &\quad =s\Big[T(s(T(u))v+s(T(v))u)\Big]\beta(w)+s(T(\beta(w)))(s(T(u))v+s(T(v))u)\\
    &\quad =s(T(u)\cdot T(v))\beta(w)+s(\alpha T(w))s(T(u))v+s(\alpha T(w))s(T(v))u\\
    &\quad =s(T(u)\cdot T(v))\beta(w)+s(T(w)\cdot T(u))\beta(v)+s(T(w)\cdot T(v))\beta(u),~~(\textsl{by}~\eqref{Cond1})\\
   &\beta(u)\diamond(v\diamond w)=\beta(u)\diamond (s(T(v))w+s(T(w))v)\\
    &\quad=s(\alpha(T(u))(s(T(v))w+s(T(w))v)+s\Big[T(s(T(v))w+s(T(w))v)\Big]\beta(u)\\
    &\quad=s(\alpha T(u))s(T(v))w+s(\alpha T(u))s(T(w))v+s(T(v)\cdot T(w))\beta(u)\\
    &\quad=s(T(u)\cdot T(v))\beta(w)+s(T(u)\cdot T(w))\beta(v)+s(T(v)\cdot T(w))\beta(u).
    \quad \mbox{(by \eqref{Cond1})}
\end{align*}
Then, $(u\diamond v)\diamond \beta(w)= \beta(u)\diamond(v\diamond w)$.
\end{proof}
\begin{cor}
Let $(A, \cdot, \alpha)$ be a commutative Hom-associative algebra, and $(s, \beta, V) $ be a bimodule.
Let $ T : V \rightarrow A $ be an $ \mathcal{O}$-operator associated to $(s, \beta, V)$. Then, $T$ is a morphism from the commutative Hom-associative algebra $(V,\diamond,\beta)$ to $(A, \cdot, \alpha)$.
\end{cor}
\begin{defn}[\cite{ShengBai:homLiebialg}]
Let $\mathcal{A}=(A, [\cdot,\cdot], \alpha)$ be a Hom-Lie algebra, and $(\rho, \beta,V)$ be a representation of $\mathcal{A}$. Then, a  linear map $ T : V \rightarrow A $
is called an $ \mathcal{O} $-operator associated to $(\rho, \beta, V)$,  if $ T $ satisfies for all
$u, v \in V$,
\begin{align}
\alpha T &= T\beta,\\
[T(u), T(v)] &= T(\rho(T(u))v - \rho(T(v))u).
\end{align}
\end{defn}
\begin{ex}
An $\mathcal{O}$-operator on a Hom-Lie algebra $\mathcal{A}=(A,[\cdot,\cdot],\alpha)$ with respect to the adjoint representation is called a Rota-Baxter operator on $\mathcal{A}$.
\end{ex}
\begin{lem}[\cite{ShengBai:homLiebialg}] \label{lemm2}
Let $T:V\rightarrow A$ be an $\mathcal{O}$-operator on a Hom-Lie algebra $(A,[\cdot,\cdot],\alpha)$ with respect to a representation $(\rho, \beta,V)$. Define a multiplication $\ast$ on $V$ by
\begin{equation}
    u\ast v=\rho(T(u)v), \text{ for all }  u,v\in V.
\end{equation}
Then, $(V,\ast,\alpha)$ is a Hom-pre-Lie algebra.
\end{lem}
\begin{cor}
Let $T:V\rightarrow A$ be an $\mathcal{O}$-operator on a Hom-Lie algebra $(A,[\cdot,\cdot],\alpha)$ with respect to a representation $(\rho, \beta,V)$. Then, $T$ is a morphism from the Hom-Lie algebra
$(V, [\cdot,\cdot]_C, \beta)$ to $(A,[\cdot,\cdot],\alpha)$,
where $[u,v]_C=u\ast v-v\ast u,~$ for all $u,v\in V.$
\end{cor}
\begin{defn}
Let $(A,\cdot, \{\cdot,\cdot\}, \alpha)$ be a transposed Hom-Poisson algebra, and let $(s,\rho, \beta,V)$ be a representation of $(A,\cdot, \{\cdot,\cdot\}, \alpha)$.
A linear operator $T:V\rightarrow A$ is called an $\mathcal{O}$-operator on $A$ if $T$ is both an $\mathcal{O}$-operator on the commutative Hom-associative algebra $(A,\cdot,\alpha)$ and an $\mathcal{O}$-operator on the Hom-Lie algebra $(A,\{\cdot,\cdot\},\alpha)$.
\end{defn}
\begin{ex}
An $\mathcal{O}$-operator on a transposed Hom-Poisson algebra $\mathcal{A}=(A,\cdot,\{\cdot,\cdot\},\alpha)$ with respect the regular bimodule and the adjoint representation is called a Rota-Baxter operator on $\mathcal{A}$.
\end{ex}
\begin{thm}
Assume that $\mathcal{A}=(A,\cdot,\{\cdot,\cdot\},\alpha)$ is a transposed Hom-Poisson algebra, and $T:V\rightarrow A$ is an $\mathcal{O}$-operator on $\mathcal{A}$ with respect to the representation
$(s,\rho,\beta,V)$. Define new operations $\diamond$ and $\ast$ on $V$ by \begin{equation}\label{lara2}u\diamond v=s(T(u))v+s(T(v))u,~u\ast v=\rho(T(u))v.\end{equation}
Then, $(V,\diamond,\ast,\alpha)$ is a Hom-pre-Lie Poisson algebra.
Therefore $V$ is a transposed Hom-Poisson algebra
 as the sub-adjacent transposed Hom-Poisson algebra of this Hom-pre-Lie Poisson algebra and $T$ is a homomorphism of transposed Hom-Poisson algebras.
 Furthermore, $T(V)=\{T(v)\colon v\in V\}\subset A$ is a subalgebra of $\mathcal{A}$ and there is an induced Hom-pre-Lie Poisson algebra structure on $T(V)$ given for all $u,v\in V$ by
\begin{equation}\label{lara1}
T(u)\diamond' T(v)=T(u\diamond v),~~T(u)\ast' T(v)=T(u\ast v).\end{equation}
Moreover, $T$ is a homomorphism of Hom-pre-Lie Poisson algebras.
\end{thm}
\begin{proof}
By Lemma \ref{lem silv double} and Lemma \ref{lemm2}, we deduce that $(A,\cdot,\alpha)$ is a commutative Hom-associative algebra and $(A,\ast,\alpha)$ is a Hom-pre-Lie algebra. Next, we prove the  axiom
\eqref{eq:pre-Poisson 1}.
For any $u,v,w\in V$,
\begin{align}
    (u\diamond v)\ast\beta(w)=&(s(T(u))v+s(T(v))u)\ast\beta(w)\cr
    =&\rho\Big(T(s(T(u))v+s(T(v))u)\Big)\beta(w)\cr
    =&\rho(T(u)\cdot T(v))\beta(w),\cr
    \beta(u)\diamond(v\ast w)=&\beta(u)\diamond(\rho(T(v))w)\cr
    =&s(T(\beta(u)))\rho(T(v))w+s(T(\rho(T(v))w))\beta(u)\nonumber,
\end{align}
Hence, we get \eqref{eq:pre-Poisson 1} by \eqref{morphismm}, \eqref{ismail.1} and \eqref{ismail.2}.
Now, we prove the axiom \eqref{eq:pre-Poisson 2},
 \begin{align*}
     (u\ast v)\diamond\beta(w)-(v\ast u)\diamond\beta(w)&=(\rho(T(u))v)\diamond\beta(w)-(\rho(T(v))u)\diamond\beta(w)\\
&=s(T(\rho(T(u))v)\beta(w)+s(T(\beta(w)))\rho(T(u))v\\
&-s(T(\rho(T(v))u)\beta(w)+s(T(\beta(w)))\rho(T(v))u\\
&=s(T(\rho(T(u))v-T(\rho(T(v))u)\beta(w)\\
&+s(T(\beta(w))\rho(T(u))v-s(T(\beta(w))\rho(T(v))u\\
&=s(\{T(u),T(v)\})\beta(w)+s(T(\beta(w))\rho(T(u))v\\&-s(T(\beta(w))\rho(T(v))u,\\
\beta(u)\ast(v\diamond w)-\beta(v)\ast(u\diamond w)&=\beta(u)\ast(s(T(v))w+s(T(w))v)\\
&-\beta(v)\ast(s(T(u))w+s(T(w))u)\\
&=\rho(T(\beta(u)))(s(T(v))w+s(T(w))v)\\
&-\rho(T(\beta(v)))(s(T(u))w+s(T(w))u)\\
&=\rho(T(\beta(u)))s(T(v))w+\rho(T(\beta(u)))s(T(w))v\\
&-\rho(T(\beta(v)))s(T(u))w+\rho(T(\beta(v)))s(T(w))u.
 \end{align*}
 Hence, we get \eqref{eq:pre-Poisson 2} by \eqref{morphismm}, \eqref{ismail.1} and \eqref{ismail.2}.

On the other hand,
\begin{align*}
T(\{u,v\}_C)=&T(u\ast v-v\ast u)
=T(\rho(T(u))v-\rho(T(v))u)=\{T(u),T(v)\},\\
T(u\diamond v)=&T(s(T(u))v+s(T(v))u)=T(u)\cdot T(v),
\end{align*}
which implies that $T$ is a homomorphism of transposed Hom-Poisson algebras. The other conclusions follow immediately.
\end{proof}
\begin{cor}
Let $(A,\cdot,\{\cdot,\cdot\},\alpha)$ be a transposed Hom-Poisson algebra. There is a Hom-pre-Lie Poisson algebra structure on $(A,\cdot,\{\cdot,\cdot\},\alpha)$ such that its sub-adjacent transposed Hom-Poisson algebra is exactly $(A,\cdot,\{\cdot,\cdot\},\alpha)$ if and only if there exists an invertible $\mathcal{O}$-operator on $(A,\cdot,\{\cdot,\cdot\},\alpha).$
\end{cor}
\begin{proof}
Suppose that there exists an invertible $\mathcal{O}$-operator $T:V\rightarrow A$ associated to the representation $(l,r,\rho, V)$, then the compatible Hom-pre-Lie
Poisson algebra structure on $A$, for all $x,y\in A$ is given by
$$x\cdot y=T(s(x)T^{-1}(y)+s(y)T^{-1}(x)),\;\;x\ast y=T(\rho(x)T^{-1}(y)).$$

Conversely, let $(A,\cdot,\ast,\alpha)$ be a Hom-pre-Lie Poisson algebra, and $(A,\cdot,\{\cdot,\cdot\},\alpha)$ be the sub-adjacent transposed Hom-Poisson algebra.
Then, the identity map $id$ is an $\mathcal{O}$-operator on $(A,\cdot,\ast,\alpha)$ with respect to the regular representation $(S, ad,\alpha,A)$.
\end{proof}

\begin{ex}
  Let $(A,\cdot,\{\cdot,\cdot\},\alpha)$ be a transposed Hom-Poisson algebra and $R:A\longrightarrow
A$ a Rota-Baxter operator. Define new operations on $A$ by
$$x\diamond y=R(x)\cdot y+x\cdot R(y),\quad x\ast y=\{R(x),y\}.$$
Then, $(A,\diamond,\ast,\alpha)$ is a Hom-pre-Lie Poisson algebra and $R$ is a homomorphism from the sub-adjacent transposed Hom-Poisson algebra $(A,\diamond,\{\cdot,\cdot\}_C,\alpha)$ to $(A,\cdot,\{\cdot,\cdot\},\alpha)$, where $\{x,y\}_C=x\ast y-y\ast x$.
\end{ex}

\section{Acknowledgement} 
Support from The Royal Swedish Academy of Sciences Foundations is gratefully acknowledged.


\begin{thebibliography}{99}

\bibitem{AbdaouiMabroukMakhlouf}
Abdaoui, E., Mabrouk, S., Makhlouf, A.: Rota-Baxter Operators on Pre-Lie Superalgebras, Bulletin of the Malaysian Math. Sci. Soc., 1-40 (2017)

\bibitem{AgoreMilitaru2015:JacobiPoissonalgs} Agore, A., Militaru, G.: Jacobi and Poisson algebras, J. Noncomm. Geom. \textbf{9}, 1295-1342 (2015)

\bibitem{Arn78} Arnold, V. I.: Mathematical Methods of Classical Mechanics, Graduate Texts Math. \textbf{60}, Springer-Verlag, New York-Heidelberg (1978)

\bibitem{AmmarEjbehiMakhlouf:homdeformation}
Ammar, F., Ejbehi, Z., Makhlouf, A.: Cohomology and deformations of Hom-algebras,  J. Lie Theory \textbf{21}(4), 813-836 (2011)

\bibitem{AttanLaraiedh:2020ConstrBihomalternBihomJordan}
Attan, S, Laraiedh, I.: Construtions and bimodules of BiHom-alternative and BiHom-Jordan algebras, 	arXiv:2008.07020 [math.RA] (2020)

\bibitem{BaiBaiGuoWu2020:transpPoisNovPois3Lie}
Bai, C., Bai, R., Guo, L., Wu, Y.: Transposed Poisson algebras, Novikov-Poisson algebras and 3-Lie algebras, arXiv:2005.01110 [math.RA], 25 pp (2020)

\bibitem{Bakayoko:LaplacehomLiequasibialg}
Bakayoko, I.: Laplacian of Hom-Lie quasi-bialgebras, International Journal of Algebra, \textbf{8}(15), 713-727 (2014)

\bibitem{Bakayoko:LmodcomodhomLiequasibialg}
Bakayoko, I.: $L$-modules, $L$-comodules and Hom-Lie quasi-bialgebras, African Diaspora Journal of Mathematics, \textbf{17}, 49-64 (2014)

\bibitem{BakBan:bimodrotbaxt}
Bakayoko, I., Banagoura, M.: Bimodules and Rota-Baxter Relations. J. Appl. Mech. Eng. \textbf{4}(5) (2015)

\bibitem{BakyokoSilvestrov:MultiplicnHomLiecoloralg}
Bakayoko, I., Silvestrov, S.: Multiplicative $n$-Hom-Lie color algebras,
In: Silvestrov, S., Malyarenko, A., Ran\u{c}i\'{c}, M. (Eds.), Algebraic Structures and Applications, Springer Proceedings in Mathematics and Statistics \textbf{317}, Ch. 7, 159-187, Springer (2020). (arXiv:1912.10216[math.QA] (2019))

\bibitem{BakyokoSilvestrov:HomleftsymHomdendicolorYauTwi}
Bakayoko, I., Silvestrov, S.:
Hom-left-symmetric color dialgebras, Hom-tridendri\-form color algebras and Yau's twisting generalizations, Afr. Mat. (2021). \\
\url{https://doi.org/10.1007/s13370-021-00871-z}. arXiv:1912.01441[math.RA]

\bibitem{BalinskiiNovikov85} Balinskii, A. A., Novikov, S. P.: Poisson brackets of hydrodynamic
type, Frobenius algebras and Lie algebras, Soviet Math. Dokl. \textbf{32}, 228-231 (1985)

\bibitem{BLLM2017} Bell, J., Launois, S., Le\'on S\'anchez, O., Moosa, R.:
Poisson algebras via model theory and differential algebraic geometry,
J. Eur. Math. Soc. \textbf{19}, 2019-2049 (2017)

\bibitem{BenMakh:Hombiliform}
Benayadi, S., Makhlouf, A.: Hom-Lie algebras with symmetric invariant nondegenerate bilinear forms, J. Geom. Phys. \textbf{76}, 38-60 (2014)

\bibitem{BenAbdeljElhamdKaygorMakhl201920GenDernBiHomLiealg}
Ben Abdeljelil, A., Elhamdadi, M., Kaygorodov, I., Makhlouf, A.: Generalized Derivations of $n$-BiHom-Lie algebras, In: Silvestrov, S., Malyarenko, A., Ran\u{c}i\'{c}, M. (Eds.), Algebraic Structures and Applications,  Springer Proceedings in Mathematics and Statistics \textbf{317}, Ch. 4, 81-97, Springer (2020). (arXiv:1901.09750[math.RA] (2019))

\bibitem{BenHassineChtiouiMabroukNcib19:CohomLiedeformBiHomleftsym}
Ben Hassine, A., Chtioui, T., Mabrouk, S., Ncib, O.: Cohomology and linear deformation of BiHom-left-symmetric algebras, arXiv:1907.06979 [math.RA], 19 pp (2019)

\bibitem{BhaskaraViswanath88} Bhaskara, K., Viswanath, K.: Poisson Algebras and Poisson Manifolds, Longman
Scientific \& Technical (1988)

\bibitem{BokutChenZhang2016} Bokut, L., Chen Y., Zhang, Z.:
On free Gelfand-Dorfman-Novikov-Poisson algebras and a PBW theorem,
J. Algebra, \textbf{500}, 153-170 (2016)

\bibitem{CaenGoyv:MonHomHopf}
Caenepeel, S., Goyvaerts, I.: Monoidal Hom-Hopf Algebras,  Comm. Algebra, \textbf{39}(6), 2216-2240 (2011)

\bibitem{CantariniKac2007} Cantarini, N., Kac, V.: Classification of linearly
compact simple Jordan and generalized Poisson superalgebras,
J. Algebra, \textbf{313}, 100-124 (2007)

\bibitem{ChariPressley94:GuideQGr}
Chari, V., Pressley, A. N.: A Guide to Quantum Groups, Cambridge Univ. Press, Cambridge, (1994)

\bibitem{ChtiouiMabroukMakhlouf1}
Chtioui, T., Mabrouk, S., Makhlouf, A.:
BiHom-alternative, BiHom-Malcev and BiHom-Jordan algebras." Rocky Mountain J. Math. \textbf{50}(1), 69-90 (2020) 

\bibitem{ChtiouiMabroukMakhlouf2}
Chtioui, T., Mabrouk, S., Makhlouf, A.:
BiHom-pre-alternative algebras and BiHom-alternative quadri-algebras,  Bull. Math. Soc. Sci. Math. Roumanie. \textbf{63 (111)}(1), 3-21 (2020)

\bibitem{DassoundoSilvestrov2021:NearlyHomass}
Dassoundo, M. L., Silvestrov, S.: Nearly associative and nearly Hom-associative algebras and bialgebras, arXiv:2101.12377 [math.RA], 24pp (2021)

\bibitem{Dirac64} Dirac, P. A. M.: Lectures on Quantum Mechanics, Belfer Grad. Sch. Sci. Monogr. Ser., Yeshive University, New York (1964)

\bibitem{Drinfeld87}
Drinfeld, V. G.: Quantum groups, in: Proc. Internat. Congr. Math. (Berkeley, 1986), AMS, Providence, RI, 798-820 (1987)

\bibitem{Dzhumadildaev2002:identderivJacalg}
Dzhumadildaev, A. S.: Identities and derivations for Jacobian algebras,
in ``Quantization, Poisson brackets and beyond (Manchester, 2001)", Contemp. Math. \textbf{315}, Amer. Math. Soc., Providence, RI, 245-278 (2002)

\bibitem{EbrahimiFardGuo08}
Ebrahimi-Fard, K., Guo, L.: Rota-Baxter algebras and dendriform algebras,
J. Pure Appl. Algebra, \textbf{212}(2), 320-339 (2008)

\bibitem{Frenkel}
Frenkel, E., Ben-Zvi, D.: Vertex Algebras and Algebraic Curves, Mathamatical Surveys and Monographs, \textbf{88}, 2nd ed., AMS, Providence, RI (2004)

\bibitem{Fresse}
Fresse, B.:  Th\'{e}orie des op\'{e}rades de Koszul et homologie des alg\`{e}bres de Poisson, Ann. Math. Blaise Pascal, \textbf{13}, 237-312 (2006)

\bibitem{GelfandDorfman79} Gel'fand, I. M., Dorfman, I. Ya.:
Hamiltonian operators and algebraic structures related to them, Funct. Anal. Appl. \textbf{13}, 248-262 (1979)

\bibitem{Gerstenhaber63}
Gerstenhaber, M.: The cohomology structure of an associative ring, Ann. Math. \textbf{78},
267-288 (1963)

\bibitem{Gerstenhaber64} Gerstenhaber, M.:
On the deformation of rings and algebras, Ann. Math. \textbf{79}, 59-103 (1964)

\bibitem{GinzburgKaledin04} Ginzburg, V., Kaledin, D.: Poisson deformations of symplectic quotient singularities, Adv. Math. \textbf{186}, 1-57 (2004)

\bibitem{GozeRemm2008:Poisalgnonassalg} Goze, M., Remm, E.: Poisson algebras in terms of non-associative algebras, J. Algebra, \textbf{320}, 294-317 (2008)

\bibitem{GrabowskiMarmo99:NambuPoissNambuJacbrackets} Grabowski, J., Marmo, G.: Remarks on Nambu-Poisson and Nambu-Jacobi bracket, J. Phys. A: Math. Gen. \textbf{32}, 4239-4247 (1999)

\bibitem{GrMakMenPan:Bihom1}
Graziani, G., Makhlouf, A., Menini, C., Panaite, F.: BiHom-Associative Algebras, BiHom-Lie Algebras and BiHom-Bialgebras, Symmetry, Integrability Geom.: Methods Appl. (SIGMA), \textbf{11}(086), 34 pp (2015)

\bibitem{GurseyTze96}
G\"{u}rsey, F., Tze, C.-H.: On the role of division, Jordan and related algebras in particle physics, \emph{World Scientific} (1996)

\bibitem{HassanzadehShapiroSutlu:CyclichomolHomasal}
Hassanzadeh, M., Shapiro, I., S{\"u}tl{\"u}, S.: Cyclic homology for Hom-associative algebras, J. Geom. Phys. \textbf{98}, 40-56 (2015)

\bibitem{HartwigLarSil:defLiesigmaderiv}
Hartwig, J. T., Larsson, D., Silvestrov, S. D.:
Deformations of Lie algebras using $\sigma$-derivations, J. Algebra, \textbf{295},  314-361 (2006)
(Preprint in Mathematical Sciences 2003:32, LUTFMA-5036-2003, Centre for Mathematical Sciences, Department of Mathematics, Lund Institute of Technology, 52 pp. (2003))


\bibitem{HounkonnouDassoundo:centersymalgbialg}
Hounkonnou, M. N., Dassoundo, M. L.: Center-symmetric Algebras and Bialgebras: Relevant Properties and Consequences. In: Kielanowski P., Ali S., Bieliavsky P., Odzijewicz A., Schlichenmaier M., Voronov T. (eds) Geometric Methods in Physics. Trends in Mathematics. 2016, pp. 281-293. Birkh{\"a}user, Cham (2016)

\bibitem{HounkonnouHoundedjiSilvestrov:DoubleconstrbiHomFrobalg}
Hounkonnou, M. N., Houndedji, G. D., Silvestrov, S.: Double constructions of biHom-Frobenius algebras, arXiv:2008.06645 [math.QA], 45pp (2020)

\bibitem{HounkonnouDassoundo:homcensymalgbialg}
Hounkonnou, M. N., Dassoundo, M. L.: Hom-center-symmetric algebras and bialgebras.  arXiv:1801.06539  [math.RA], 19 pp (2018)

\bibitem{Huebschmann90}
Huebschmann, J.: Poisson cohomology and quantization, J. Reine Angew. Math. \textbf{408}, 57-113 (1990)

\bibitem{KarasevMaslov:NonlinPoissonbrquantiz}
Karasev, N. V., Maslov, V. P.: Nonlinear Poisson brackets, geometry and quantization,
Transl. Math. Monographs, vol. 119, Amer. Math. Soc., Providence, RI, 1993, xi + 366 pp.

\bibitem{Kirillov76}
Kirillov, A.: Local Lie algebras, Uspekhi Mat. Nauk, \textbf{31:4}(190), 57-76 (1976) (Russian).
(Engl. Transl.: Russian Math. Surveys, \textbf{31:4}, 55-75 (1976))

\bibitem{kms:narygenBiHomLieBiHomassalgebras2020}
Kitouni, A., Makhlouf, A., Silvestrov, S.: On $n$-ary generalization of BiHom-Lie algebras and BiHom-associative algebras, In: Silvestrov, S., Malyarenko, A., Ran\u{c}i\'{c}, M. (Eds.), Algebraic Structures and Applications, Springer Proceedings in Mathematics and Statistics \textbf{317}, Springer, Ch. 5, 99-126 (2020)

\bibitem{Kontsevich03}
Kontsevich, M.: Deformation quantization of Poisson manifolds, Lett. Math. Phys. \textbf{66}, 157-216 (2003)

\bibitem{KosmannSchwarzbach96} Kosmann-Schwarzbach, Y.: From Poisson to Gerstenhaber algebras, Ann. Inst. Fourier, \textbf{46}, 1243-1274 (1996)

\bibitem{Laraiedh1:2021:BimodmtchdprsBihomprepois}
Laraiedh, I.: Bimodules and matched pairs of noncommutative BiHom-(pre)-Poisson algebras, arXiv:2102.11364 [math.RA], 30 pp (2021)

\bibitem{LarssonSigSilvJGLTA2008:QuasiLiedefFttN}
Larsson, D., Sigurdsson, G., Silvestrov, S. D.: Quasi-Lie deformations on the algebra $\mathbb{F}[t]/(t^N)$, J. Gen. Lie Theory Appl. \textbf{2}(3), 201-205 (2008)

\bibitem{LarssonSilvJA2005:QuasiHomLieCentExt2cocyid}
Larsson, D., Silvestrov, S. D.: Quasi-Hom-Lie algebras, central extensions and $2$-cocycle-like identities, J. Algebra, \textbf{288}, 321-344 (2005) (Preprints in Mathematical Sciences 2004:3, LUTFMA-5038-2004, Centre for Mathematical Sciences, Department of Mathematics, Lund Institute of Technology, Lund University (2004))

\bibitem{LarssonSilv:quasiLiealg}
Larsson, D., Silvestrov, S. D.: Quasi-Lie algebras. In
"Noncommutative Geometry and Representation Theory in Mathematical Physics". Contemp. Math., 391, Amer. Math. Soc., Providence, RI, 241-248 (2005)
(Preprints in Mathematical Sciences 2004:30, LUTFMA-5049-2004, Centre for Mathematical Sciences, Department of Mathematics, Lund Institute of Technology, Lund University (2004))

\bibitem{LSGradedquasiLiealg}
Larsson, D., Silvestrov, S. D.: Graded quasi-Lie agebras, Czechoslovak J. Phys. \textbf{55}, 1473-1478 (2005)

\bibitem{LarssonSilv:quasidefsl2}
Larsson, D., Silvestrov, S. D.: Quasi-deformations of $sl_2(\mathbb{F})$ using twisted derivations, Comm. Algebra, \textbf{35}, 4303-4318 (2007)
(Preprint in Mathematical Sciences 2004:26, LUTFMA-5047-2004, Centre for Mathematical Sciences, Lund Institute of Technology, Lund University (2004). arXiv:math/0506172 [math.RA] (2005))

\bibitem{LarssonSilvestrovGLTMPBSpr2009:GenNComplTwistDer}
Larsson, D., Silvestrov, S. D.: On generalized $N$-complexes comming from twisted derivations, In: Silvestrov, S., Paal, E., Abramov, V., Stolin, A. (Eds.),
Generalized Lie Theory in Mathematics, Physics and Beyond, Springer-Verlag, Ch. 7, 81-88 (2009)

\bibitem{Li99} Li, L.: Classical $r$-matrices and compatible Poisson structures for Lax equations on Poisson algebras, Comm. Math. Phys. \textbf{203}, 573-592 (1999)

\bibitem{Lichnerowicz77} Lichnerowicz, A.: Les vari\'{e}ti\'{e}s de Poisson et leurs alg\`{e}bras de Lie associ\'{e}es, J. Diff. Geom. \textbf{12}, 253-300 (1977)

\bibitem{LiuMakhMenPan:RotaBaxteropsBiHomassalg}
Liu, L., Makhlouf, A., Menini, C., Panaite, F.:  Rota-Baxter operators on BiHom-associative algebras and related structures, arXiv:1703.07275 [math.RA], 27pp (2017).

\bibitem{LodayVallette2012} Loday, J., Vallette, B.: Algebraic Operads,
Grundlehren der Mathematischen Wissenschaften, \textbf{346}, Springer (2012)

\bibitem{MaMakhSil:CurvedOoperatorSyst}
Ma, T., Makhlouf, A., Silvestrov, S.:
Curved $\mathcal{O}$-operator systems, arXiv: 1710.05232 [math.RA], 17pp (2017)

\bibitem{MaMakhSil:RotaBaxbisyscovbialg}
Ma, T., Makhlouf, A., Silvestrov, S.:
Rota-Baxter bisystems and covariant bialgebras, arXiv:1710.05161[math.RA], 30pp (2017)

\bibitem{MaMakhSil:RotaBaxCosyCoquasitriMixBial}
Ma, T., Makhlouf, A., Silvestrov, S.:
Rota-Baxter Cosystems and Coquasitriangular Mixed Bialgebras, J. Algebra Appl. \textbf{20}(04), 2150064 (2021) 

\bibitem{MaZheng:RotaBaxtMonoidalHomAlg}
Ma, T., Zheng, H.: Some results on Rota-Baxter monoidal Hom-algebras, Results Math. \textbf{72} (1-2), 145-170 (2017).

\bibitem{MabroukNcibSilvestrov2020:GenDerRotaBaxterOpsnaryHomNambuSuperalgs}
Mabrouk, S., Ncib, O., Silvestrov, S.: Generalized Derivations and Rota-Baxter Operators of $n$-ary Hom-Nambu Superalgebras, Adv. Appl. Clifford Algebras, \textbf{31}, 32, (2021). (arXiv:2003.01080[math.QA]) 

\bibitem{MakhloufHomdemdoformRotaBaxterHomalg2011}
Makhlouf, A.: Hom-dendriform algebras and Rota-Baxter Hom-algebras, In: Bai, C., Guo, L., Loday, J.-L. (eds.), Nankai Ser. Pure Appl. Math. Theoret. Phys., \textbf{9}, World Sci. Publ. 147-171 (2012)

\bibitem{Makhl:HomaltHomJord}
Makhlouf, A.: Hom-alternative algebras and Hom-Jordan algebras, Int. Elect. J. Alg., \textbf{8}, 177-190 (2010). (arXiv:0909.0326 (2009))

\bibitem{Makhlouf2010:ParadigmnonassHomalgHomsuper}
Makhlouf, A.: Paradigm of nonassociative Hom-algebras and Hom-superalgebras,
Proceedings of Jordan Structures in Algebra and Analysis Meeting, 145-177 (2010).
(arXiv:1001.4240v1)

\bibitem{ms:homstructure}
Makhlouf, A., Silvestrov, S. D.:
Hom-algebra structures. J. Gen. Lie Theory Appl. \textbf{2}(2), 51--64 (2008)
(Preprints in Mathematical Sciences  2006:10, LUTFMA-5074-2006, Centre for Mathematical Sciences, Department of Mathematics, Lund Institute of Technology, Lund University (2006))

\bibitem{MakhSil:HomHopf}
Makhlouf, A., Silvestrov, S.:
Hom-Lie admissible Hom-coalgebras and Hom-Hopf algebras,
In: Silvestrov, S., Paal, E., Abramov, V., Stolin, A. (Eds.),
Generalized Lie Theory in Mathematics, Physics and Beyond, Springer-Verlag, Berlin, Heidelberg, Ch. 17, 189-206 (2009) (Preprints in Mathematical Sciences, Lund University, Centre for Mathematical Sciences, Centrum Scientiarum Mathematicarum (2007:25) LUTFMA-5091-2007 and in arXiv:0709.2413 [math.RA] (2007))

\bibitem{MakhSilv:HomAlgHomCoalg}
Makhlouf, A., Silvestrov, S. D.:
Hom-algebras and Hom-coalgebras, J. Algebra Appl. \textbf{9}(04), 553-589 (2010) (Preprints in Mathematical Sciences, Lund University, Centre for Mathematical Sciences, Centrum Scientiarum Mathematicarum, (2008:19) LUTFMA-5103-2008. arXiv:0811.0400
[math.RA] (2008)) 

\bibitem{MakhSilv:HomDeform}
Makhlouf, A., Silvestrov, S.: Notes on $1$-parameter formal deformations of Hom-associative and Hom-Lie algebras, Forum Math. \textbf{22}(4), 715-739 (2010)
(Preprints in Mathematical Sciences, Lund University, Centre for Mathematical Sciences, Centrum Scientiarum Mathematicarum, (2007:31) LUTFMA-5095-2007. arXiv:0712.3130v1 [math.RA] (2007))

\bibitem{MakYau:RotaBaxterHomLieadmis}
Makhlouf, A., Yau, D.: Rota-Baxter Hom-Lie admissible algebras,
Comm. Alg., \textbf{23}(3), 1231-1257 (2014)

\bibitem{MarklRemm2006} Markl, M., Remm, E.: Algebras with one operation including Poisson and other Lie-admissible algebras, J. Algebra, \textbf{299}, 171-189 (2006)

\bibitem{MishchenkoPetrogrRegev2007:PoisPIalg} Mishchenko, S. P., Petrogradsky, V. M., Regev, A.:
Poisson PI algebras, Trans. Amer. Math. Soc. \textbf{359}, 4669-4694 (2007)

\bibitem{OdA}
Odzijewicz, A.: Hamiltonian and quantum mechanics, Geom. Topol. Monogr. \textbf{17}, 385-472 (2011)

\bibitem{RichardSilvestrovJA2008}
Richard, L., Silvestrov, S. D.: Quasi-Lie structure of $\sigma$-derivations of $\mathbb{C}[t^{\pm1}]$,
J. Algebra, \textbf{319}(3), 1285-1304 (2008) (arXiv:math/0608196[math.QA] (2006). Preprints in mathematical sciences (2006:12), LUTFMA-5076-2006, Centre for Mathematical Sciences, Lund University (2006))

\bibitem{RichardSilvestrovGLTbnd2009}
Richard, L., Silvestrov, S. D.:
A note on quasi-Lie and Hom-Lie structures of $\sigma$-derivations of
${\mathbb C}[z_1^{\pm 1},\ldots,z_n^{\pm 1}]$, In: Silvestrov, S., Paal, E., Abramov, V., Stolin, A. (Eds.), Generalized Lie Theory in Mathematics, Physics and Beyond, Springer-Verlag, Ch. 22, 257-262, (2009)

\bibitem{Rinehart63} Rinehart, G.: Differential forms for general commutative algebras,
Trans. Amer. Math. Soc. \textbf{108}, 195-222 (1963)

\bibitem{Pol97}
Polishchuk, A.: Algebraic geometry of Poisson brackets, J. Math. Sci. \textbf{84}, 1413-1444 (1997)

\bibitem{SaadaouSilvestrov:lmgderivationsBiHomLiealgebras}
Saadaou, N, Silvestrov, S.: On $(\lambda,\mu,\gamma)$-derivations of BiHom-Lie algebras,	arXiv:2010.09148 [math.RA], (2020)

\bibitem{Sheng:homrep}
Sheng, Y.: Representations of Hom-Lie algebras, Algebr. Reprensent. Theory \textbf{15}, 1081-1098 (2012)

\bibitem{ShengBai:homLiebialg}
Sheng, Y.,  Bai,  C.: A  new  approach  to  Hom-Lie  bialgebras,  J. Algebra,  \textbf{399},  232-250 (2014)

\bibitem{SigSilv:CzechJP2006:GradedquasiLiealgWitt}
Sigurdsson, G., Silvestrov, S.: Graded quasi-Lie algebras of Witt type, Czechoslovak J. Phys. \textbf{56}, 1287-1291 (2006)

\bibitem{SigSilv:GLTbdSpringer2009}
Sigurdsson, G., Silvestrov, S.: Lie color and Hom-Lie algebras of Witt type and their central extensions, In: Silvestrov, S., Paal, E., Abramov, V., Stolin, A. (Eds.), Generalized Lie Theory in Mathematics, Physics and Beyond, Springer-Verlag, Berlin, Heidelberg, Ch. 21, 247-255 (2009)

\bibitem{SilvestrovParadigmQLieQhomLie2007}
Silvestrov, S.: Paradigm of quasi-Lie and quasi-Hom-Lie algebras and quasi-defor\-mations, In "New techniques in Hopf algebras and graded ring theory", K. Vlaam. Acad. Belgie Wet. Kunsten (KVAB), Brussels, 165-177 (2007)

\bibitem{SilvestrovZardeh2021:HNNextinvolmultHomLiealg}
Silvestrov, S., Zargeh, C.: HNN-extension of involutive multiplicative Hom-Lie algebras, arXiv:2101.01319 [math.RA], 14pp (2021)

\bibitem{QSunHomPrealtBialg}
Sun, Q.: On Hom-Prealternative Bialgebras, Algebr. Represent. Theor. \textbf{19}, 657-677 (2016)

\bibitem{SunLi2017:parakahlerhomhomleftsymetric}
Sun, Q., Li, H.: On parak\"{a}hler hom-Lie algebras and hom-left-symmetric bialgebras, Comm. Algebra,  \textbf{45}(1), 105-120 (2017)

\bibitem{Weinstein77} Weinstein, A.: Lecture on Symplectic Manifolds, CBMS Regional Conference Series in Mathematics, \textbf{29}, Amer. Math. Soc., Providence, R. I. (1979)

\bibitem{Vaisman1} Vaisman, I.: Lectures on the Geometry of Poisson Manifolds, Progr. Math.
\textbf{118}, Birkh\"{a}user Verlag, Basel (1994)

\bibitem{Xu:noncomPoissonalg94}
Xu, P.: Noncommutative Poisson algebras. Amer. J. Math. \textbf{116}, 101-125 (1994)

\bibitem{Xu:NovikovPoissonalg97}
Xu, X.: Novikov-Poisson algebras, J. Algebra, \textbf{190}, 253-279 (1997)

\bibitem{Yau:HomEnv}
Yau, D.: Enveloping algebras of Hom-Lie algebras, J. Gen. Lie Theory Appl. \textbf{2}(2), 95-108 (2008). (arXiv:0709.0849 [math.RA] (2007))

\bibitem{Yau:ModuleHomalg}
Yau, D.: Module Hom-algebras, arXiv:0812.4695[math.RA], 10pp (2008)

\bibitem{Yau:HomYangBaHomLiequasitribial}
Yau  D.:  The  Hom-Yang-Baxter  equation,  Hom-Lie  algebras  and  quasi-triangular  bialgebras,  J. Phys. A.: Math. Theor. \textbf{42}(16), 165-202  (2009)

\bibitem{Yau:HomHom}
Yau, D.: Hom-algebras and homology, J. Lie Theory \textbf{19}(2), 409-421 (2009)

\bibitem{Yau:HombialgcomoduleHomalg}
Yau, D.: Hom-bialgebras and comodule Hom-algebras, Int. Electron.~J. Algebra, \textbf{8}, 45-64  (2010). (arXiv:0810.4866[math.RA] (2008))

\bibitem{YauHomMalcevHomalternHomJord}
Yau, D.: Hom-Malcev, Hom-alternative, and Hom-Jordan algebras, Int. Electron. J. Algebra, \textbf{11},  177-217 (2012)

\end{thebibliography}
\end{document}